\documentclass[a4paper]{amsart}
\usepackage[marginparwidth=0pt,margin=20truemm]{geometry} 
\usepackage{ascmac,amssymb,amsmath,amsthm}
\usepackage{mathtools}
\usepackage{graphicx,color}
\usepackage{tikz}
\usetikzlibrary{arrows,positioning,fit,calc,intersections,through,backgrounds,quotes,angles}
\usepackage{thmtools}
\usepackage{stmaryrd}
\usepackage[hidelinks]{hyperref}

\usepackage{graphicx,xcolor}
\usepackage{ascmac}
\usepackage{amsmath,amssymb,amsthm,amsfonts,amscd}
\usepackage{braket}

\usetikzlibrary{cd}

\renewcommand{\epsilon}{\varepsilon}

\newcommand{\ds}{\displaystyle}
\theoremstyle{plain}
\newtheorem{thm}{Theorem}[section]

\newtheorem{prop}[thm]{Proposition}
\newtheorem{cor}[thm]{Corollary}

\newtheorem{exam}[thm]{Example}

\theoremstyle{definition}
\newtheorem{ass}[thm]{Assumption}
\newtheorem{dfn}[thm]{Definition}

\theoremstyle{remark}
\newtheorem{rem}[thm]{Remark}

\numberwithin{equation}{section}

\title{Holonomy preserving transformations of weighted graphs\\
and its application to knot theory}    
\author{Atsuhide Nagasaka}

\begin{document}

\begin{abstract}
Goda~\cite{H-Goda} showed that the twisted Alexander polynomial can be recovered from the zeta function of a matrix-weighted graph. Motivated by this, we study transformations of weighted graphs that preserve this zeta function, introducing a notion of holonomy as an analogy for the accumulation of weights along cycles.

We extend the framework from matrices to group elements, and show that holonomy preserving transformations correspond to transformations of group presentations and preserve the twisted Alexander polynomial from a graph-theoretic viewpoint. We also generalize to quandle-related structures, where the holonomy condition coincides with the Alexander pair condition of Ishii and Oshiro~\cite{Ishii-Oshiro}. This perspective allows us to view knot diagrams as covering-like structures enriched with holonomy.
\end{abstract}

\maketitle

\tableofcontents

\section{Introduction}

The twisted Alexander polynomial, introduced by Lin~\cite{Lin} and later formulated using group presentations by Wada~\cite{M-Wada}, is a powerful knot invariant defined from a group presentation and a representation. It generalizes the classical Alexander polynomial, introduced by J.~W.~Alexander~\cite{Alexander}, and can distinguish knots that other invariants cannot. For example, it separates the Kinoshita--Terasaka knot, the Conway knot, and the unknot using a representation into $SL_2(\mathbb{F}_7)$.

Goda~\cite{H-Goda} showed that the twisted Alexander polynomial can be recovered as the reciprocal of a zeta function associated with a matrix-weighted graph. This matrix-weighted zeta function, introduced by Watanabe and Fukumizu~\cite{Watanabe-Fukumizu}, is defined by
\[
    \zeta_G(w) = \prod_{[C]} \det(I - w(C))^{-1},
\]
where the product is over equivalence classes of cycles $[C]$ in the graph $G$, and $w(C)$ denotes the product of matrices along $C$. Goda showed that for a knot diagram $K_{D^*}$ with a base point and a representation $\rho: \pi_K \to GL_n(\mathbb{C}[t^{\pm1}])$, one has
\[
    \Delta_{K_{D^*}, \rho}(t) = \zeta_{G,\rho}(t)^{-1},
\]
where $\zeta_{G,\rho}(t)$ denotes the matrix-weighted zeta function with weights given by $\rho$.

The matrix zeta function can be viewed as a generating function for a notion of {\it holonomy}---an accumulation of weights along cycles---in a graph. From this viewpoint, Goda generalized the known relationship between the Ihara zeta function~\cite{Ihara} and the Alexander polynomial.

To understand the invariance of the twisted Alexander polynomial via this framework, we introduce {\it holonomy preserving transformations} on weighted graphs. These transformations preserve the matrix-weighted zeta function:
\begin{thm}
  \[
  (G,w) \sim (G',w') \Longrightarrow \zeta_G(w) = \zeta_{G'}(w').
  \]
  \hfill{$\square$}
\end{thm}

We then generalize this setup by extending weights from matrices to group elements under certain conditions. We define transformations of group-weighted graphs and relate them to transformations of group presentations.

\begin{thm}
    Let $(X,R,B), (X',R',B')$ be two presentations of a group $G$ that are strongly Tietze equivalent preserving base points. Then, $\Gamma(X,R,B) \sim \Gamma(X',R',B')$. \hfill{$\square$}
\end{thm}

\begin{cor}
  Let $(X,R,B)$ be a finitely presented group, and $\Gamma(X,R,B)$ be a group-weighted graph made by $(X,R,B)$. We construct a $\Gamma$ which is equivalent to $\Gamma(X,R,B)$ through a finite sequence of elementary transformations. Then there exists a finitely presented group that satisfies the following properties:
  \begin{itemize}
    \item $\Gamma = \Gamma(X', R', B')$.
    \item $(X,R)$ and $(X',R')$ are strongly Tietze equivalent.\hfill{$\square$}
  \end{itemize}
\end{cor}

As a result, we provide a graph-theoretic proof of the invariance of the twisted Alexander polynomial. Finally, we generalize the notion of weights further to algebraic data arising from quandles. We show that holonomy preserving transformations in this setting correspond to {\it Alexander pairs}, a notion introduced by Ishii and Oshiro~\cite{Ishii-Oshiro}.

\begin{thm}
    The following are equivalent.
    \begin{itemize}
        \item[(1)] $g=(g_{1}^{+}, g_{1}^{-}, g_{2}^{+}, g_{2}^{-})$ preserves holonomy, that is, (A) through (C) holds.
        \item[(2)] $g=(g_{1}^{+}, g_{1}^{-}, g_{2}^{+}, g_{2}^{-})$ is f-twisted Alexander weight.\hfill{$\square$}
    \end{itemize}
\end{thm}

\section*{Acknowledgements}
I would like to express my deepest gratitude to my parent for their unwavering support. I am also deeply grateful to Associate Professor Akishi Kato for his invaluable advice and insightful comments. Additionally, I would like to thank Kohei Takehira, Masahiro Watanabe and Yuya Murakami for their helpful discussions and feedback.

\section{Matrix-weighted graphs and their transformations}
In this section, we introduce the notion of a matrix-weighted zeta function of a matrix-weighted graph $G$, following \cite{Mat-Wt-Graph-Intro}. We then define transformations of a matrix-weighted graph that preserve the zeta function.

Let $G = (V(G), E(G))$ be a connected, directed, and finite graph, possibly with multiple edges and loops, where $V(G)$ and $E(G)$ denote the sets of vertices and oriented edges, respectively. If an oriented edge $e \in E(G)$ is from $a \in V(G)$ to $b \in V(G)$, we denote $a = o(e)$ and $b = t(e)$.

A {\it path} $P$ of length $n$ in $G$ is a sequence $P = (e_{1}, \ldots, e_{n})$ of $n$ arcs, $e_{i} \in D(G), t(e_{i}) = o(e_{i+1})\ (1\leq i \leq n-1)$. The indices are considered modulo $n$. A path $P$ is called a {\it cycle} when $o(e_{1}) = t(e_{n})$. 

We define an equivalence relation between two cycles. Two cycles $C_{1} = (e_{1}, \ldots, e_{n})$ and $C_{2} = (f_{1}, \ldots, f_{n})$ are called {\it equivalent} if there exists an integer $k$ such that $f_{j} = e_{j+k}$ holds for all $j$. We denote the equivalence class containing a cycle $C$ as $[C]$. Let $C^{r}$ represent the cycle obtained by traversing around cycle $C$ for $r$ times, which is called a {\it power} of $C$. A cycle $C$ is said to be {\it prime} if it cannot be represented as a power of a strictly smaller cycle. 

Then, we define the matrix-weighted zeta function of a graph $G$. Suppose $G$ has $n$ vertices $v_{1}, \ldots, v_{n}$ and $m$ edges. Let $(a_{1}, \ldots, a_{n}) \in (\mathbb{Z}_{>0})^{n}$ and set $a_{v_{i}} = a_{i}\ (1 \leq i \leq n)$. For each edge $e$ from $v_{i}$ to $v_{j}$, let $w = w(v_{i}, v_{j})$ be an $a_{i} \times a_{j}$ matrix over $\mathbb{C}[t^{\pm 1}]$\footnote{
Generally, weights need not be elements of $\mathbb{C}[t^{\pm 1}]$, but due to our primary focus on studying a relation between zeta function and twisted Alexander polynomial in \S 4, we consider the weights to be elements of $\mathbb{C}[t^{\pm 1}]$.
}. The set $\{w(e) | e\in E(G)\}$ is called the matrix-weight of $G$, and $(G,w)$ is called the {\it matrix-weighted graph}.

\begin{center}
  \begin{tikzpicture}[x=1.4cm,y=1.4cm]
   \node (A1) at (-1.5,0) {};
   \node (A2) at (1.5,0) {};
   \draw (A1) node [below left] {$v_{i}$};
   \draw (A2) node [below right] {$v_{j}$};
   \draw (A1) node [above] {};
   \draw (A2) node [above] {};

   \fill (A1) circle [radius=1pt];
   \fill (A2) circle [radius=1pt];

   \draw[->] (A1) to node[below]{$
   \displaystyle
   \begin{bmatrix}
     a_{1,1}& \cdots & a_{1,d_{j}} \\
     \vdots & \ddots & \vdots \\
     a_{d_{i},1} & \cdots & a_{d_{i},d_{j}}
   \end{bmatrix}
   $} (A2);

  \end{tikzpicture}
\end{center}

\begin{dfn}[{\cite{Watanabe-Fukumizu}}] \label{Def-Of-zeta}
    For each cycle $C = (e_{1}, \ldots, e_{n})$, we define the weight of the cycle $C$ as the products of the weights of the edges:
    \[
        w(C) = w(e_{1}) \ldots w(e_{n}).
    \]
    The {\it matrix-weighted zeta function} $\zeta_{G}(w)$ of $(G,w)$ is defined by
    \begin{align}
        \zeta_{G}(w) = \prod_{[C]} \det(I - w(C))^{-1},
    \end{align} \label{Def-Of-zeta-siki}
    where $[C]$ runs over all equivalence classes of prime cycles of $G$.\hfill{$\square$}
\end{dfn} 

The term $I - w(C)$ in (\ref{Def-Of-zeta-siki}) represents the transformation of basis when traversing a cycle $C$. Here, {\it holonomy} does not refer to the classical notion of parallel transport in differential geometry but rather serves as an analogy for the evolution of weights (or representations) along graph cycles. In this sense, the matrix-weighted zeta function can be regarded as a generating function that encodes the holonomy of matrix-weighted graphs. We now define transformations of matrix-weighted graphs, preserving their holonomy.

\begin{dfn} \label{Def-MatWtGraphTrfm}
  For a matrix-weighted graph, the following transformations apply to a matrix-weighted graph, modifying a part of the graph as described below or reversing all edges in the diagram. We call these operations {\it elementary transformations of a matrix-weighted graph}. In the following figures, edges without arrows can be directed in both directions.

  \begin{itemize}
    \item[(1)] (change of basis)
    
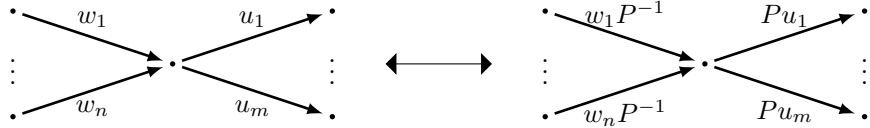
\begin{center}
      \begin{tikzpicture}[x=1.4cm,y=1.4cm]

        \node (A1) at (-0.5,0.5) {};
        \node (Ai) at (-0.5,0) {$\vdots$};
        \node (An) at (-0.5,-0.5) {};
        \node (B1) at (2.5,0.5) {};
        \node (Bi) at (2.5,0) {$\vdots$};
        \node (Bn) at (2.5,-0.5) {};
        \node (O) at (1,0) {};

        \fill (A1) circle [radius=1pt];
        \fill (An) circle [radius=1pt];
        \fill (O) circle [radius=1pt];
        \fill (B1) circle [radius=1pt];
        \fill (Bn) circle [radius=1pt];

        \draw[line width=1pt, arrows={ - latex}] (A1) to node[above]{$w_{1}$} (O);
        \draw[line width=1pt, arrows={ - latex}] (An) to node[below]{$w_{n}$} (O);
        \draw[line width=1pt, arrows={ - latex}] (O) to node[above]{$u_{1}$} (B1);
        \draw[line width=1pt, arrows={ - latex}] (O) to node[below]{$u_{m}$} (Bn);

        \draw[arrows={triangle 90-triangle 90}] (3,0) -- (4,0);

        \node (A1') at (4.5,0.5)  {};
        \node (Ai') at (4.5,0)  {$\vdots$};
        \node (An') at (4.5,-0.5) {};
        \node (B1') at (7.5,0.5) {};
        \node (Bi') at (7.5,0) {$\vdots$};
        \node (Bn') at (7.5,-0.5) {};
        \node (O') at (6,0) {};

        \fill (A1') circle [radius=1pt];
        \fill (An') circle [radius=1pt];
        \fill (O') circle [radius=1pt];
        \fill (B1') circle [radius=1pt];
        \fill (Bn') circle [radius=1pt];

        \draw[line width=1pt, arrows={ - latex}] (A1') to node[above]{$w_{1}P^{-1}$} (O');
        \draw[line width=1pt, arrows={ - latex}] (An') to node[below]{$w_{n}P^{-1}$} (O');
        \draw[line width=1pt, arrows={ - latex}] (O') to node[above]{$Pu_{1}$} (B1');
        \draw[line width=1pt, arrows={ - latex}] (O') to node[below]{$Pu_{m}$} (Bn');

      \end{tikzpicture}
    \end{center}

    \item[(2)] (null edge)
    \begin{center}
      \begin{tikzpicture}[x=1.4cm,y=1.4cm]
        \node (A) at (0,0) {};
        \node (B) at (1,0) {};

        \fill (A) circle [radius=1pt];
        \fill (B) circle [radius=1pt];


        \draw[arrows={triangle 90-triangle 90}] (1.5,0) -- (2.5,0);

        \node (A') at (3,0) {};
        \node (B') at (4,0) {};

        \fill (A') circle [radius=1pt];
        \fill (B') circle [radius=1pt];

        \draw[line width=1pt] (A') to node[above]{$0$} (B');

      \end{tikzpicture}
    \end{center}

    \item[(3)] (summing weights)
    
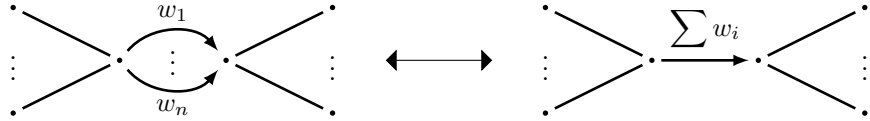
\begin{center}
      \begin{tikzpicture}[x=1.4cm,y=1.4cm]
        \node (A1) at (-1,0.5) {};
        \node (Ai) at (-1,0) {$\vdots$};
        \node (An) at (-1,-0.5) {};
        \node (B1) at (2,0.5) {};
        \node (Bi) at (2,0) {$\vdots$};
        \node (Bn) at (2,-0.5) {};
        \node (O1) at (0,0) {};
        \node (O2) at (1,0) {};
        \node (wi) at (0.5,0.05) {$\vdots$};

        \fill (A1) circle [radius=1pt];
        \fill (An) circle [radius=1pt];
        \fill (O1) circle [radius=1pt];
        \fill (O2) circle [radius=1pt];
        \fill (B1) circle [radius=1pt];
        \fill (Bn) circle [radius=1pt];

        \draw[line width=1pt] (A1) -> (O1);
        \draw[line width=1pt] (An) -> (O1);
        \draw[line width=1pt] (O2) -> (B1);
        \draw[line width=1pt] (O2) -> (Bn);
        \draw[line width=1pt, arrows={ - latex}] (O1) to [out=50, in=130]  node[above]{$w_{1}$} (O2);
        \draw[line width=1pt, arrows={ - latex}] (O1) to [out=-50, in=-130] node[below]{$w_{n}$} (O2);

        \draw[arrows={triangle 90-triangle 90}] (2.5,0) -- (3.5,0);

        \node (A1') at (4,0.5) {};
        \node (Ai') at (4,0) {$\vdots$};
        \node (An') at (4,-0.5) {};
        \node (B1') at (7,0.5) {};
        \node (Bi') at (7,0) {$\vdots$};
        \node (Bn') at (7,-0.5) {};
        \node (O1') at (5,0) {};
        \node (O2') at (6,0) {};

        \fill (A1') circle [radius=1pt];
        \fill (An') circle [radius=1pt];
        \fill (O1') circle [radius=1pt];
        \fill (O2') circle [radius=1pt];
        \fill (B1') circle [radius=1pt];
        \fill (Bn') circle [radius=1pt];

        \draw[line width=1pt] (A1') -> (O1');
        \draw[line width=1pt] (An') -> (O1');
        \draw[line width=1pt, arrows={ - latex}] (O1') to node[above]{$\ds \sum w_{i}$} (O2');
        \draw[line width=1pt] (O2') -> (B1');
        \draw[line width=1pt] (O2') -> (Bn');

      \end{tikzpicture}
    \end{center}

    \item[(4)] (source/sink eliminating)
    
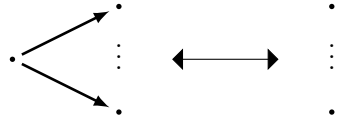
\begin{center}
      \begin{tikzpicture}[x=1.4cm,y=1.4cm]

        \node (O) at (-1,0) {};
        \node (A1) at (0,0.5) {};
        \node (Ai) at (0,0.1) {$\vdots$};
        \node (An) at (0,-0.5) {};

        \fill (A1) circle [radius=1pt];
        \fill (An) circle [radius=1pt];
        \fill (O) circle [radius=1pt];

        \draw[line width=1pt, arrows={ - latex}] (O) -> (A1);
        \draw[line width=1pt, arrows={ - latex}] (O) -> (An);

        \draw[arrows={triangle 90-triangle 90}] (0.5,0) -- (1.5,0);

        \node (A1') at (2,0.5) {};
        \node (Ai') at (2,0.1) {$\vdots$};
        \node (An') at (2,-0.5) {};


        \fill (A1') circle [radius=1pt];
        \fill (An') circle [radius=1pt];

      \end{tikzpicture}
    \end{center}

    \item[(5)] (hub vertex resolution)
    
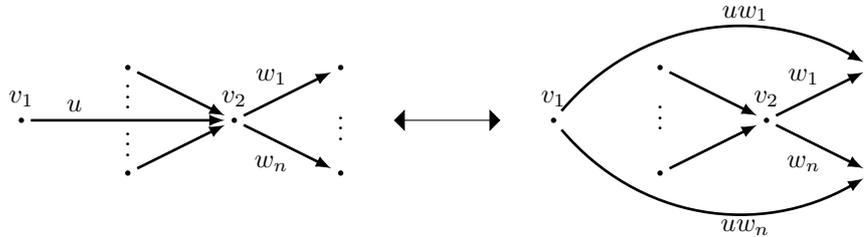
\begin{center}
      \begin{tikzpicture}[x=1.4cm,y=1.4cm]
        \node (A0) at (-1,0) {};
        \node (v1) at (-1, 0.2) {$v_{1}$};
        \node (v2) at (1, 0.2) {$v_{2}$};
        \node (A1) at (0,0.5) {};
        \node (Ai1) at (0,0.3) {$\vdots$};
        \node (Ai2) at (0,-0.15){$\vdots$};
        \node (An) at (0,-0.5) {};
        \node (B1) at (2,0.5) {};
        \node (Bi) at (2,0) {$\vdots$};
        \node (Bn) at (2,-0.5) {};
        \node (O) at (1,0) {};
        \node (u) at (-0.5,0) [above] {$u$};
        \node (w1) at (1.35,0.25) [above] {$w_{1}$};
        \node (w1) at (1.35,-0.25) [below] {$w_{n}$};
        \node (t) at (1.35,-1) {};

        \fill (A0) circle [radius=1pt];
        \fill (A1) circle [radius=1pt];
        \fill (An) circle [radius=1pt];
        \fill (O) circle [radius=1pt];
        \fill (B1) circle [radius=1pt];
        \fill (Bn) circle [radius=1pt];

        \draw[line width=1pt, arrows={ - latex}] (A0) -> (O);
        \draw[line width=1pt, arrows={ - latex}] (A1) -> (O);
        \draw[line width=1pt, arrows={ - latex}] (An) -> (O);
        \draw[line width=1pt, arrows={ - latex}] (O) -> (B1);
        \draw[line width=1pt, arrows={ - latex}] (O) -> (Bn);

        \draw[arrows={triangle 90-triangle 90}] (2.5,0) -- (3.5,0);

        \node (A0') at (4,0) {};
        \node (v1') at (4, 0.2) {$v_{1}$};
        \node (v2') at (6, 0.2) {$v_{2}$};
        \node (A1') at (5,0.5) {};
        \node (Ai') at (5,0.1) {$\vdots$};
        \node (An') at (5,-0.5) {};
        \node (B1') at (7,0.5) {};
        \node (Bi') at (7,0.1) {$\vdots$};
        \node (Bn') at (7,-0.5) {};
        \node (O') at (6,0) {};
        \node (w1') at (6.35,0.25) [above] {$w_{1}$};
        \node (w1') at (6.35,-0.25) [below] {$w_{n}$};
        \node (uw1') at (5.8,0.85) [above] {$uw_{1}$};
        \node (uwn') at (5.8,-0.85) [below] {$uw_{n}$};

        \fill (A0') circle [radius=1pt];
        \fill (A1') circle [radius=1pt];
        \fill (An') circle [radius=1pt];
        \fill (O') circle [radius=1pt];
        \fill (B1') circle [radius=1pt];
        \fill (Bn') circle [radius=1pt];

        \draw[line width=1pt, arrows={ - latex}]  (A0') to [out=50, in=150] (B1');
        \draw[line width=1pt, arrows={ - latex}]  (A0') to [out=-50, in=210] (Bn');
        \draw[line width=1pt, arrows={ - latex}] (A1') -> (O');
        \draw[line width=1pt, arrows={ - latex}] (An') -> (O');
        \draw[line width=1pt, arrows={ - latex}] (O') -> (B1');
        \draw[line width=1pt, arrows={ - latex}] (O') -> (Bn');

      \end{tikzpicture}
    \end{center}

    where $v_{1} \neq v_{2}$.
  \end{itemize}

  Two matrix-weighted graphs $(G,w), (G',w')$ are called equivalent, if they are connected by a finite sequence of elementary transformations of weighted graphs, and we denote this equivalence by $(G,w) \sim (G',w')$.\hfill{$\square$}
\end{dfn} 

\begin{exam}
  The following operation corresponds to the composition of morphisms:
    
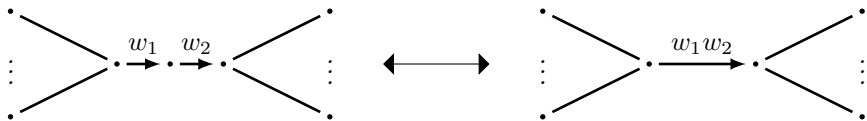
\begin{center}
        \begin{tikzpicture}[x=1.4cm,y=1.4cm]
         \node (A1) at (-1,0.5) {};
         \node (Ai) at (-1,0) {$\vdots$};
         \node (An) at (-1,-0.5) {};
         \node (B1) at (2,0.5) {};
         \node (Bi) at (2,0) {$\vdots$};
         \node (Bn) at (2,-0.5) {};
         \node (O1) at (0,0) {};
         \node (O2) at (0.5,0) {};
         \node (O3) at (1,0) {};
    
         \fill (A1) circle [radius=1pt];
         \fill (An) circle [radius=1pt];
         \fill (O1) circle [radius=1pt];
         \fill (O2) circle [radius=1pt];
         \fill (O3) circle [radius=1pt];
         \fill (B1) circle [radius=1pt];
         \fill (Bn) circle [radius=1pt];

         \draw[line width = 1pt, arrows={ - latex}] (O1) to node[above]{$w_{1}$} (O2);
         \draw[line width = 1pt, arrows={ - latex}] (O2) to node[above]{$w_{2}$} (O3);
         \draw[line width = 1pt] (A1) -> (O1);
         \draw[line width = 1pt] (An) -> (O1);
         \draw[line width = 1pt] (O3) -> (B1);
         \draw[line width = 1pt] (O3) -> (Bn);

         \draw[arrows={triangle 90-triangle 90}] (2.5,0) -- (3.5,0);
    
         \node (A1') at (4,0.5) {};
         \node (Ai') at (4,0) {$\vdots$};
         \node (An') at (4,-0.5) {};
         \node (B1') at (7,0.5) {};
         \node (Bi') at (7,0) {$\vdots$};
         \node (Bn') at (7,-0.5) {};
         \node (O1') at (5,0) {};
         \node (O2') at (6,0) {};
    
         \fill (A1') circle [radius=1pt];
         \fill (An') circle [radius=1pt];
         \fill (O1') circle [radius=1pt];
         \fill (O2') circle [radius=1pt];
         \fill (B1') circle [radius=1pt];
         \fill (Bn') circle [radius=1pt];

         \draw[line width = 1pt, arrows={ - latex}] (O1') to node[above]{$w_{1}w_{2}$} (O2');
         \draw[line width = 1pt] (A1') -> (O1');
         \draw[line width = 1pt] (An') -> (O1');
         \draw[line width = 1pt] (O2') -> (B1');
         \draw[line width = 1pt] (O2') -> (Bn');
    
        \end{tikzpicture}
    \end{center}
    can be obtained as a combination of the following fundamental operations. By setting $n=1$ in (4) and assuming that the number of input and output edges at the central vertex is exactly one each, we obtain:
    \begin{center}
        \begin{tikzpicture}[x=1.4cm,y=1.4cm]
        \node (A) at (-1,0) {};
        \node (B) at (1,0) {};
        \node (O) at (0,0) {};

        \fill (A) circle [radius=1pt];
        \fill (O) circle [radius=1pt];
        \fill (B) circle [radius=1pt];
        
        \draw[line width = 1pt, arrows={ - latex}] (A) to node[above]{$w_{1}$} (O);
        \draw[line width = 1pt, arrows={ - latex}] (O) to node[above]{$w_{2}$} (B);

        \draw[arrows={triangle 90-triangle 90}] (1.5,0) to node [above] {(5)} (2.5,0);
        
        \node (A') at (3,0) {};
        \node (B') at (5,0) {};
        \node (O') at (4,0) {};

        \fill (A') circle [radius=1pt];
        \fill (O') circle [radius=1pt];
        \fill (B') circle [radius=1pt];

        \draw[line width = 1pt, arrows={ - latex}] (O') to node[above]{$w_{2}$} (B');
        \draw[line width = 1pt, arrows={ - latex}] (A') to [out=50, in=130, looseness = 1.3] node[above]{$w_{1}w_{2}$} (B');
        
        \draw[arrows={triangle 90-triangle 90}] (5.5,0) to node [above] {(4)}  (6.5,0);
        
        \node (A'') at (7,0) {};
        \node (B'') at (9,0) {};
        
        \fill (A'') circle [radius=1pt];
        \fill (B'') circle [radius=1pt];
        
        \draw[line width = 1pt, arrows={ - latex}] (A'') to [out=50, in=130, looseness = 1.3] node[above]{$w_{1}w_{2}$} (B'');
        
        \end{tikzpicture}
    \end{center}
    \hfill{$\square$}
\end{exam}

\begin{thm} \label{Thm-TrfmPreserveZeta}
  \[
  (G,w) \sim (G',w') \Longrightarrow \zeta_{G}(w) = \zeta_{G'}(w')
  \]
\end{thm} 

\begin{proof}
    The statements for (1), (2), and (4) in Definition \ref{Def-MatWtGraphTrfm} are obvious, so we discuss (3) and (5) below. For (5), we label the arcs as follows:
    
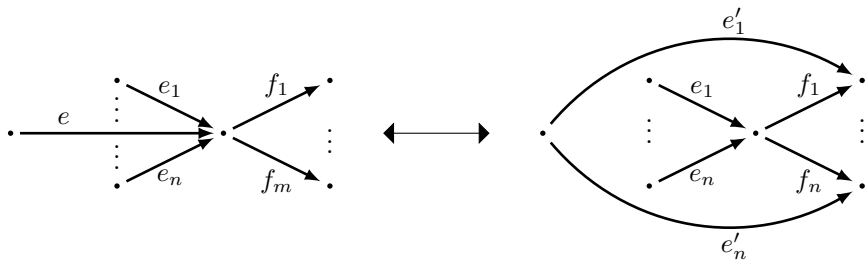
\begin{center}
        \begin{tikzpicture}[x=1.4cm,y=1.4cm]
            \node (A0) at (-1,0) {};
            \node (A1) at (0,0.5) {};
            \node (Ai1) at (0,0.3) {$\vdots$};
            \node (Ai2) at (0,-0.15){$\vdots$};
            \node (An) at (0,-0.5) {};
            \node (B1) at (2,0.5) {};
            \node (Bi) at (2,0) {$\vdots$};
            \node (Bn) at (2,-0.5) {};
            \node (O) at (1,0) {};
            \node (u) at (-0.5,0) [above] {$e$};
            \node (t) at (1.35,-1) {};
            
            \fill (A0) circle [radius=1pt];
            \fill (A1) circle [radius=1pt];
            \fill (An) circle [radius=1pt];
            \fill (O) circle [radius=1pt];
            \fill (B1) circle [radius=1pt];
            \fill (Bn) circle [radius=1pt];

            \draw[line width = 1pt, arrows={ - latex}] (A0) -> (O);
            \draw[line width = 1pt, arrows={ - latex}] (A1) to node[above]{$e_{1}$} (O);
            \draw[line width = 1pt, arrows={ - latex}] (An) to node[below]{$e_{n}$} (O);
            \draw[line width = 1pt, arrows={ - latex}] (O) to node[above]{$f_{1}$} (B1);
            \draw[line width = 1pt, arrows={ - latex}] (O) to node[below]{$f_{m}$} (Bn);

            \draw[arrows={triangle 90-triangle 90}] (2.5,0) -- (3.5,0);
            
            \node (A0') at (4,0) {};
            \node (A1') at (5,0.5) {};
            \node (Ai') at (5,0.1) {$\vdots$};
            \node (An') at (5,-0.5) {};
            \node (B1') at (7,0.5) {};
            \node (Bi') at (7,0.1) {$\vdots$};
            \node (Bn') at (7,-0.5) {};
            \node (O') at (6,0) {};
            \node (uw1') at (5.8,0.85) [above] {$e'_{1}$};
            \node (uwn') at (5.8,-0.85) [below] {$e'_{n}$};
            
            \fill (A0') circle [radius=1pt];
            \fill (A1') circle [radius=1pt];
            \fill (An') circle [radius=1pt];
            \fill (O') circle [radius=1pt];
            \fill (B1') circle [radius=1pt];
            \fill (Bn') circle [radius=1pt];
            
            \draw[line width = 1pt, arrows={ - latex}]  (A0') to [out=50, in=150] (B1');
            \draw[line width = 1pt, arrows={ - latex}]  (A0') to [out=-50, in=210] (Bn');
            \draw[line width = 1pt, arrows={ - latex}] (A1') to node[above]{$e_{1}$} (O');
            \draw[line width = 1pt, arrows={ - latex}] (An') to node[below]{$e_{n}$} (O');
            \draw[line width = 1pt, arrows={ - latex}] (O') to node[above]{$f_{1}$} (B1');
            \draw[line width = 1pt, arrows={ - latex}] (O') to node[below]{$f_{n}$} (Bn');
        
        \end{tikzpicture}
    \end{center}
    Then, the cycle
    \[
    (\cdots, e, f_{i}, \cdots) \longleftrightarrow (\cdots, e_{i}', \cdots)
    \]
    forms a one-to-one correspondence, the cycles remain prime, and the weight of each cycle is preserved. Thus, the matrix-weighted zeta function is preserved. For (3), this follows from Theorem \ref{Thm-IharaBassDetFomula}.
\end{proof}

\begin{thm}[{\cite{Watanabe-Fukumizu}}]\label{Thm-IharaBassDetFomula}
  We define the adjacency matrix $A(G,w)$ as follows. Let $V(G) = \{v_{1}, \ldots, v_{n}\}$ be the set of vertices, and let $A(G,w)$ be a matrix whose components are defined as below:
  \begin{align*}
    A(G,w)_{ij}:=
    \begin{cases}
      0\ \ \ \ &(\text{if there exist no edges from } v_{i} \text{ to } v_{j})\\
      \text{sum of weights from } v_{i} \text{ to } v_{j} &(\text{if there exist edges from } v_{i} \text{ to } v_{j})
    \end{cases}
  \end{align*}
  Then the following holds.
  \begin{align}
  \label{Thm-IharaBass-siki}
    \zeta_{G}(w) = \det (I-A(G,w))^{-1}
  \end{align}
\end{thm}

\begin{proof}
  Firstly, we calculate the left-hand side of (\ref{Thm-IharaBass-siki}):
  \begin{align*}
    \log \zeta_{G}(w)
    &=\log \prod_{[C]} \det(I - w(C))^{-1}\\
    &=\sum_{[C]} \mathrm{tr}( \log (I - w(C))^{-1})\\
    &=\sum_{[C]} \sum_{k \geq 1} \frac{1}{k}\mathrm{tr} (w(C)^{k})
  \end{align*}
  There are $|C|$ ways to take points of cycle $C$:
  \begin{align*}
    \sum_{[C]} \sum_{k \geq 1} \frac{1}{k}\mathrm{tr} (w(C)^{k})
    &= \sum_{C : \text{prime cycle}} \frac{1}{|C|} \sum_{k \geq 1}\frac{1}{k}\mathrm{tr} (w(C)^{k})\\
    &= \sum_{C : \text{prime cycle}}  \sum_{k \geq 1}\frac{1}{k|C|}\mathrm{tr} (w(C)^{k}).
  \end{align*}
  $w(C)^{k}$ is the weight of a cycle which goes $k$ times around a cycle $C$:
  \[
    \sum_{C : \text{prime cycle}}  \sum_{k \geq 1}\frac{1}{k|C|}\mathrm{tr} (w(C)^{k})
    = \sum_{C : \mathrm{cycles}} \frac{1}{|C|}\mathrm{tr} (w(C)).
  \]
  Since
  \begin{align*}
    \sum_{C : \mathrm{cycles}} \frac{1}{|C|}\mathrm{tr} (w(C))
    = \sum_{k\geq 1} \frac{1}{k}\mathrm{tr} ((A(G,w))^{k}),
  \end{align*}
  (\ref{Thm-IharaBass-siki}) is proved.
\end{proof}

\section{Group-weighted graphs and their transformations}
In this section, we construct weighted graphs from finitely presented groups satisfying some conditions by using free differential and define their transformation.

\subsection{Free differential and base points on relations}
Let $F_{n}$ be a free group generated by $x_{1}, \ldots, x_{n}$, that is, $F_{n} = \braket{x_{1}, \ldots, x_{n}}$. The {\it free differential} is a homomorphism
\[
  \frac{\partial}{\partial x_{i}} : \mathbb{Z}F_{n} \to \mathbb{Z}F_{n}
\]
satisfying the following properties:
\begin{itemize}
  \item for all $p,q \in F_{n}$
    \[
      \frac{\partial}{\partial x_{i}}(pq) = \frac{\partial}{\partial x_{i}}(p) + p\frac{\partial}{\partial x_{i}}(q).
    \]
  \item $\ds \frac{\partial}{\partial x_{i}} (x_{j}) = \delta_{ij}$.
\end{itemize}

\begin{dfn}
  Let $\Braket{X|R} = \Braket{x_{1}, \ldots, x_{n} | r_{1}, \ldots, r_{m}}$ be a presentation of a group $G$. Within a relation $r$, we choose and fix one $x_{i}$ as the {\it base point of $r$}. \hfill{$\square$}
\end{dfn}

We make the following assumption.
\begin{ass} \label{Ass-BasePt}
  Each $r_{i}$ has a base point and the alphabets (generators) used in $X$ are pairwise distinct. Furthermore, let $\mathrm{pr} : F_{s} \twoheadrightarrow G$ be a projection, then
  \begin{align}
    \mathrm{pr} \left( \frac{\partial r_{i}}{\partial x_i}\right) \neq 0 \in \mathbb{Z}G \label{Ass-siki}
  \end{align}
  is satisfied.\hfill{$\square$}
\end{ass}

\begin{rem}
    Geometrically, Assumption~\ref{Ass-BasePt} is analogous to the implicit function theorem, in that it ensures the existence of a ``smooth'' coordinate structure in the non-commutative group ring $\mathbb{Z}G$ around the chosen base points. This condition suggests the presence of meaningful invariants, which we investigate via group-weighted graphs. Notably, many knot groups—having deficiency one (i.e., the number of generators exceeds that of relations by one)—satisfy this assumption, making them a natural fit for our approach based on group-weighted graphs.\hfill{$\square$}
\end{rem}

If a presentation $\Braket{X|R}$ includes a base point for each relation $r \in R$, it is denoted as  $(X,R,B)$. Generally, there are multiple ways to select base points, but when choosing a pair of base points, we write the relation as
\[
  r_{i} = x_{i}f_{i}(x_{1}, \ldots , x_{n})^{-1}
\]
by appropriately labeling relations. We further rewrite this as
\[
  r_{i} : x_{i} = f_{i}
\]
in the following discussion. According to (\ref{Ass-siki}), 
\[
  \mathrm{pr}\left( \frac{\partial f_{i}}{\partial x_{i}}\right) \neq 1 \in \mathbb{Z}G.
\]
If \( f_i^{-1} \) involves \( x_i \), then there is an ambiguity in choosing which occurrence of \( x_i \) in \( r_i \) should be treated as the base point. For the purpose of construction, we fix one such choice. Later, we show that the resulting zeta function is independent of this choice up to units in \( \mathbb{C}[t^{\pm1}] \).

\begin{rem} \label{Rem-assump-inj}
  Assumption \ref{Ass-BasePt} means that there exists an injection
  \[
  f : R \to X,
  \]
  and for each $r\in R$, $f(r)$ is an alphabet of $r$ and the specific occurrence of $f(r)$ in $r$ is determined.\hfill{$\square$}
\end{rem}

\subsection{Construction of group-weighted graphs and their transformations}
\begin{dfn}\label{Def-GrpWtGraph}
  For a finitely presented group $\Braket{x_{1}, \ldots, x_{n} |r_{1}, \ldots, r_{m} }$ satisfying Assumption \ref{Ass-BasePt}, we regard $r_{i} = x_{i}f_{i}^{-1}$ as $x_{i} = f_{i}$ and take a free differential of both sides:
  \begin{align}
  \label{Def-GrpWtGraph-siki}
    dx_{i} = df_{i} = \sum_{j=1}^{n} \frac{\partial f_{i}}{\partial x_{j}}dx_{j}.
  \end{align}
  We construct a graph with $n$ vertices $v_{1}, \ldots, v_{n}$. For each vertex, we make the edges from $v_{i}$ to $v_{j}$ using a relation $r_{i}$ and define its weight as
  \[
    \mathrm{pr}\left( \frac{\partial f_{i}}{\partial x_{j}}\right),
  \]
  
  where $\mathrm{pr} : \Braket{x_{1}, \ldots, x_{n}} \twoheadrightarrow \Braket{x_{1}, \ldots, x_{n} |r_{1}, \ldots, r_{m} }$. We call it the {\it group-weighted graph}, which is denoted by $\Gamma(X,R,B)$. \hfill{$\square$}
\end{dfn}

\begin{rem}
    When calculating the weights on group-weighted graphs, what we are interested in is not $\partial f_{i}/\partial x_{j}$ but rather in $\mathrm{pr}(\partial f_{i}/\partial x_{j})$. To simplify the notation, we write it as
    \[
        dx_{i} = \sum_{j}\mathrm{pr}\left( \frac{\partial f_{i}}{\partial x_{j}}\right)dx_{j}.
    \]
    For example, if the relation $r_{i}$ is given by $x_{i} = x_{j}x_{i+1}x_{j}^{-1}$, then the corresponding differential equation for constructing a group-weighted graph is $dx_{i} = x_{j}dx_{i+1} + (1-x_{i})dx_{j}$.\hfill{$\square$}
\end{rem}

We define the elementary transformation of group-weighted graphs.

\begin{dfn}\label{Def-GrpWtGraphTrfm}
  For a group-weighted graph, we consider the following local transformations that modify parts of a group-weighted graph, as described below. We call their transformations the {\it elementary transformation of a group-weighted graph}. In the following figures, edges without arrows can be directed in both directions.
  
  \begin{itemize}

    \item[(G1)] (null edge)

    \begin{center}
      \begin{tikzpicture}[x=1.4cm,y=1.4cm]
        \node (A) at (0,0) {};
        \node (B) at (1,0) {};

        \fill (A) circle [radius=1pt];
        \fill (B) circle [radius=1pt];


        \draw[arrows={triangle 90-triangle 90}] (1.5,0) -- (2.5,0);

        \node (A') at (3,0) {};
        \node (B') at (4,0) {};

        \fill (A') circle [radius=1pt];
        \fill (B') circle [radius=1pt];

        \draw[line width=1pt, arrows={ - latex}] (A') to node[above]{$0$} (B');

      \end{tikzpicture}
    \end{center}
    where the origin of the edge with weight $0$ is not a sink.

    \item[(G2)] (summing weights)
    
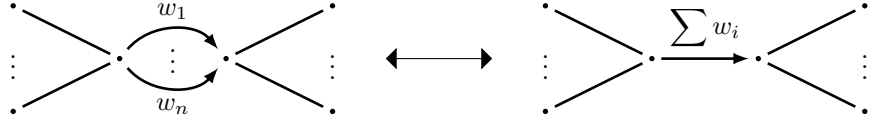
\begin{center}
      \begin{tikzpicture}[x=1.4cm,y=1.4cm]
        \node (A1) at (-1,0.5) {};
        \node (Ai) at (-1,0) {$\vdots$};
        \node (An) at (-1,-0.5) {};
        \node (B1) at (2,0.5) {};
        \node (Bi) at (2,0) {$\vdots$};
        \node (Bn) at (2,-0.5) {};
        \node (O1) at (0,0) {};
        \node (O2) at (1,0) {};
        \node (wi) at (0.5,0.05) {$\vdots$};

        \fill (A1) circle [radius=1pt];
        \fill (An) circle [radius=1pt];
        \fill (O1) circle [radius=1pt];
        \fill (O2) circle [radius=1pt];
        \fill (B1) circle [radius=1pt];
        \fill (Bn) circle [radius=1pt];

        \draw[line width=1pt] (A1) -> (O1);
        \draw[line width=1pt] (An) -> (O1);
        \draw[line width=1pt] (O2) -> (B1);
        \draw[line width=1pt] (O2) -> (Bn);
        \draw[line width=1pt, arrows={ - latex}] (O1) to [out=50, in=130]  node[above]{$w_{1}$} (O2);
        \draw[line width=1pt, arrows={ - latex}] (O1) to [out=-50, in=-130] node[below]{$w_{n}$} (O2);

        \draw[arrows={triangle 90-triangle 90}] (2.5,0) -- (3.5,0);

        \node (A1') at (4,0.5) {};
        \node (Ai') at (4,0) {$\vdots$};
        \node (An') at (4,-0.5) {};
        \node (B1') at (7,0.5) {};
        \node (Bi') at (7,0) {$\vdots$};
        \node (Bn') at (7,-0.5) {};
        \node (O1') at (5,0) {};
        \node (O2') at (6,0) {};

        \fill (A1') circle [radius=1pt];
        \fill (An') circle [radius=1pt];
        \fill (O1') circle [radius=1pt];
        \fill (O2') circle [radius=1pt];
        \fill (B1') circle [radius=1pt];
        \fill (Bn') circle [radius=1pt];

        \draw[line width=1pt] (A1') -> (O1');
        \draw[line width=1pt] (An') -> (O1');
        \draw[line width=1pt, arrows={ - latex}] (O1') to node[above]{$\ds \sum w_{i}$} (O2');
        \draw[line width=1pt] (O2') -> (B1');
        \draw[line width=1pt] (O2') -> (Bn');

      \end{tikzpicture}
    \end{center}

    \item[(G3)] (source eliminating)
    
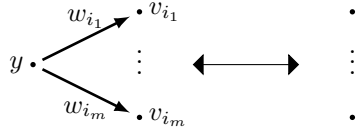
\begin{center}
      \begin{tikzpicture}[x=1.4cm,y=1.4cm]

        \node (O) at (-1,0) {};
        \node (A1) at (0,0.5) {};
        \node (Ai) at (0,0.1) {$\vdots$};
        \node (An) at (0,-0.5) {};

        \fill (A1) circle [radius=1pt];
        \fill (An) circle [radius=1pt];
        \fill (O) circle [radius=1pt];
        \draw (A1) node [right] {$v_{i_{1}}$};
        \draw (An) node [right] {$v_{i_{m}}$};
        \draw (O) node [left] {$y$};

        \draw[line width=1pt, arrows={ - latex}] (O) to node[above]{$w_{i_{1}}$} (A1);
        \draw[line width=1pt, arrows={ - latex}] (O) to node[below]{$w_{i_{m}}$} (An);

        \draw[arrows={triangle 90-triangle 90}] (0.5,0) -- (1.5,0);

        \node (A1') at (2,0.5) {};
        \node (Ai') at (2,0.1) {$\vdots$};
        \node (An') at (2,-0.5) {};


        \fill (A1') circle [radius=1pt];
        \fill (An') circle [radius=1pt];

      \end{tikzpicture}

    \end{center}
    Assuming that for $w_{i_{1}}, \ldots, w_{i_{m}}$ shown in the figure, there exists a word $f$ satisfying 
    \[
      df = \sum_{j} w_{i_{j}}dx_{i_{j}}.
    \]

    \item[(G4)] (hub vertex resolution)
    
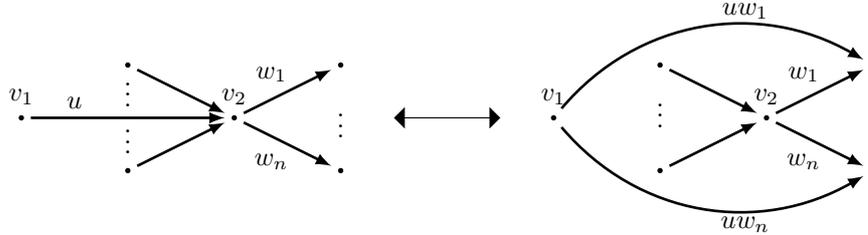
\begin{center}
      \begin{tikzpicture}[x=1.4cm,y=1.4cm]
        \node (A0) at (-1,0) {};
        \node (v1) at (-1, 0.2) {$v_{1}$};
        \node (v2) at (1, 0.2) {$v_{2}$};
        \node (A1) at (0,0.5) {};
        \node (Ai1) at (0,0.3) {$\vdots$};
        \node (Ai2) at (0,-0.15){$\vdots$};
        \node (An) at (0,-0.5) {};
        \node (B1) at (2,0.5) {};
        \node (Bi) at (2,0) {$\vdots$};
        \node (Bn) at (2,-0.5) {};
        \node (O) at (1,0) {};
        \node (u) at (-0.5,0) [above] {$u$};
        \node (w1) at (1.35,0.25) [above] {$w_{1}$};
        \node (w1) at (1.35,-0.25) [below] {$w_{n}$};
        \node (t) at (1.35,-1) {};

        \fill (A0) circle [radius=1pt];
        \fill (A1) circle [radius=1pt];
        \fill (An) circle [radius=1pt];
        \fill (O) circle [radius=1pt];
        \fill (B1) circle [radius=1pt];
        \fill (Bn) circle [radius=1pt];

        \draw[line width=1pt, arrows={ - latex}] (A0) -> (O);
        \draw[line width=1pt, arrows={ - latex}] (A1) -> (O);
        \draw[line width=1pt, arrows={ - latex}] (An) -> (O);
        \draw[line width=1pt, arrows={ - latex}] (O) -> (B1);
        \draw[line width=1pt, arrows={ - latex}] (O) -> (Bn);

        \draw[arrows={triangle 90-triangle 90}] (2.5,0) -- (3.5,0);

        \node (A0') at (4,0) {};
        \node (v1') at (4, 0.2) {$v_{1}$};
        \node (v2') at (6, 0.2) {$v_{2}$};
        \node (A1') at (5,0.5) {};
        \node (Ai') at (5,0.1) {$\vdots$};
        \node (An') at (5,-0.5) {};
        \node (B1') at (7,0.5) {};
        \node (Bi') at (7,0.1) {$\vdots$};
        \node (Bn') at (7,-0.5) {};
        \node (O') at (6,0) {};
        \node (w1') at (6.35,0.25) [above] {$w_{1}$};
        \node (w1') at (6.35,-0.25) [below] {$w_{n}$};
        \node (uw1') at (5.8,0.85) [above] {$uw_{1}$};
        \node (uwn') at (5.8,-0.85) [below] {$uw_{n}$};

        \fill (A0') circle [radius=1pt];
        \fill (A1') circle [radius=1pt];
        \fill (An') circle [radius=1pt];
        \fill (O') circle [radius=1pt];
        \fill (B1') circle [radius=1pt];
        \fill (Bn') circle [radius=1pt];

        \draw[line width=1pt, arrows={ - latex}]  (A0') to [out=50, in=150] (B1');
        \draw[line width=1pt, arrows={ - latex}]  (A0') to [out=-50, in=210] (Bn');
        \draw[line width=1pt, arrows={ - latex}] (A1') -> (O');
        \draw[line width=1pt, arrows={ - latex}] (An') -> (O');
        \draw[line width=1pt, arrows={ - latex}] (O') -> (B1');
        \draw[line width=1pt, arrows={ - latex}] (O') -> (Bn');

      \end{tikzpicture}
    \end{center}
    
    where $v_{1} \neq v_{2}$.

  \end{itemize}
  \hfill{$\square$}
\end{dfn}

If there exists a representation $\rho : G \to GL_{k}(\mathbb{C}[t^{\pm 1}])$ and replace each weight $g \in \mathbb{Z}G$ by $\rho (g)$, we can construct a matrix-weighted graph from a group-weighted graph, which is denoted by $\Gamma_{\rho}(X,R,B)$ .

\begin{prop} \label{Prop-GrpEquivIsWtEquiv}
    It two group-weighted graphs $\Gamma(X,R,B)$ and $\Gamma(X', R', B')$ are connected with a finite sequence of a group-weighted graph, $\Gamma_{\rho}(X,R,B)$ and $\Gamma_{\rho}(X',R',B')$ are matrix-weighted graph equivalent.\hfill{$\square$}
\end{prop}

\begin{proof}
  It is easy to see by Definition \ref{Def-MatWtGraphTrfm} and Definition \ref{Def-GrpWtGraphTrfm}.
\end{proof}

\begin{prop} \label{Prop-ElemProperty}
  Let $\zeta_{\rho}(X,R,B)$ be a matrix-weighted zeta function of $\Gamma_{\rho}(X,R,B)$, then the followings are satisfied.
  \begin{itemize}
    \item[\textup{(1)}] $\rho \sim \rho' \Longrightarrow \zeta_{\rho}(X,R,B) = \zeta_{\rho'}(X,R,B)$
    \item[\textup{(2)}] $\rho \sim \rho_{1} \oplus \rho_{2} \Longrightarrow \zeta_{\rho}(X,R,B) = \zeta_{\rho_{1}}(X,R,B)\zeta_{\rho_{2}}(X,R,B)$\hfill{$\square$}
  \end{itemize}
\end{prop}

\begin{proof}
  \begin{itemize}
    \item[(1)] When $\rho \sim \rho'$, there exists a regular matrix $P$ such that $P^{-1}\rho(g)P = \rho'(g)$ for all $g \in G$, and then $P^{-1}\rho(w)P = \rho'(w)$ for $w \in \mathbb{Z}G$. Since 
    \[
      \det(I - w(C)) = \det (P^{-1}IP - P^{-1}w(C)P),
    \]
    $\zeta_{\rho}(X,R,B) = \zeta_{\rho'}(X,R,B)$ is satisfied.
    
    \item[(2)] By (1), we can assume $\rho = \rho_{1} \oplus \rho_{2}$. Since
    \[\det \left(I -
    \left[
      \begin{array}{c|c}
        \rho_{1}(g) & 0\\ \hline
        0 & \rho_{2}(g)
      \end{array}
      \right]
      \right)
      = \det(I - \rho_{1}(g)) \det(I - \rho_{2}(g)), 
    \]
    $\zeta_{\rho}(X,R,B) = \zeta_{\rho_{1}}(X,R,B)\zeta_{\rho_{2}}(X,R,B)$ is satisfied. 
  \end{itemize}
\end{proof}

As we mentioned above, we may have some choices which $x_{i}$ in $r_{i}$ to choose when constructing a matrix-weighted graph, but the following holds.

\begin{prop}\label{Prop-Base-Choice}
  A matrix-weighted zeta function is unique up to units in $\mathbb{C}[t^{\pm 1}]$ regardless of the labeling of $r_{i}$ or the choices of $x_{i}$ in $r_{i}$.
\end{prop}

\begin{proof} 
    According to Theorem \ref{Thm-IharaBassDetFomula}, the right-hand side of
    \[
    dx_{i}
    = d f_{i}
    = \sum_{j} \frac{\partial f_{i}}{\partial x_{j}} dx_{j}
    \]
    as given in (\ref{Def-GrpWtGraph-siki}), corresponds to the adjacency matrix $A(G,w)$. $\zeta_{\rho}(X,R,B)$ is determined up to $\pm 1$ regardless of the labeling of $r_{i}$, as labeling $r_{i}$ corresponds to row operations on $I - A(G,w)$. 
    
    Now, we examine the effect of choosing $x_{i}$ as the left-hand side of $r_{i} : x_{i} = f_{i}$. It is sufficient to consider the following two cases. We will simply write $x$ instead of $x_{i}$.
    \begin{itemize}
      \item[(a)] $x = w_{1}xw_{2} \longleftrightarrow x = w_{1}^{-1} x w_{2}^{-1}$
      \item[(b)] $x = w_{1}x^{-1}w_{2} \longleftrightarrow x = w_{2} x^{-1} w_{1}$
    \end{itemize}
    
    For the case (a), the free differential of the right-hand sides of equations are
    \begin{align*}
      d(w_{1}xw_{2})
      &= dw_{1} + w_{1}dx + w_{1}xdw_{2}, \\
      d(w_{1}^{-1}xw_{2}^{-1})
      &= -w_{1}^{-1}dw_{1} + w_{1}^{-1}dx - w_{1}^{-1}xw_{2}^{-1}dw_{2}.
    \end{align*}

    We express $w_1$ and $w_2$ as $w_{1} = A,dx + A'$ and $w_{2} = B,dx + B'$, where $A$ and $B$ are independent of $dx$. Then, the differentials can be written as follows:
    \begin{align*}
      d(w_{1}xw_{2})
      &= (A+w_{1}+w_{1}xB)dx + (A'+w_{1}xB'), \\
      d(w_{1}^{-1}xw_{2}^{-1})
      &= (-w_{1}^{-1}A+w_{1}^{-1}-w_{1}^{-1}xw_{2}^{-1}B)dx + (-w_{1}^{-1}A'-w_{1}^{-1}xw_{2}^{-1}B').
    \end{align*}
    Since
    \begin{align*}
        I - (A+w_{1}+w_{1}xB) 
        &= w_{1} \{w_{1}^{-1} - (w_{1}^{-1}A + I + xB)\}\\
        &= -w_{1} \{I - (-w_{1}^{-1}A+w_{1}^{-1}-xB)\},\\
        A'+w_{1}xB'
        &= -w_{1}(-w_{1}^{-1}A'-xB')
    \end{align*}
    and because graphs are constructed with $x=w_{1}^{-1}xw_{2}^{-1}$, the claim follows. The case (b) follows similarly.
\end{proof}

That is, if an injection $f : R \to X$ is given and each alphabet $f(r)$ in $r$ is different from each other, a zeta function induced by $f$ is determined up to units in $\mathbb{C}[t^{\pm 1}]$. If the set of base points is preserved, the zeta function is also preserved:

\begin{cor}
  Let $\Braket{X | R}$ be a presentation of a group $G$, two maps $f, g : X \to R$ be injections and each $f(r)$ and each $g(r)$ be an alphabet in $r$. If $f(R) = g(R)$, zeta functions induced by $f, g$ are equal up to units in $\mathbb{C}[t^{\pm 1}]$. 
\end{cor}

\subsection{Transformations of group presentations and that of group-weighted graphs}
Now, we discuss the relationship between the transformation of group-weighted graphs and the transformation of finitely presented groups satisfying Assumption \ref{Ass-BasePt}. As is well known, following transformations of group presentations induce the isomorphism of two groups.

\begin{dfn}[{\cite{M-Wada}}] \label{Def-StronglyTietzeTrfm}
  Two finitely presented groups $\Braket{X|R}, \Braket{X'|R'}$ are connected by a finite sequence of operations of (1), (2), (3), (4), and their inverse operations, two presentations are called {\it strongly Tietze equivalent}:
  \begin{itemize}
    \item[(1)] To replace one of the relations, $r_{i}$, by its inverse $r_{i}^{-1}$.
    \item[(2)] To replace one of the relations, $r_{i}$, by its conjugate $wr_{i}w^{-1}\ (w \in F_{n})$.
    \item[(3)] To replace one of the relations, $r_{i}$, by its inverse $r_{i}r_{k}\ (i \neq k)$.
    \item[(4)] To add a new generator $x$ and a new relation $xw^{-1}$, where $w$ is any word in $x_{1}, \ldots, x_{n}$. Thus, the resulting presentation is
    \[
        \Braket{x_{1}, \ldots, x_{n}, x | r_{1}, \ldots, r_{m}, xw^{-1}}\ \ (w \in F_{n}).
    \]
    \hfill{$\square$}
  \end{itemize}
\end{dfn} 

\begin{dfn}
  Let $(X,R,B)$ and $(X',R',B')$ be two presentations of a group $G$, and assume that $(X',R',B')$ is obtained from $(X,R,B)$ by performing operations (1) $\sim$ (4) as described in Definition \ref{Def-StronglyTietzeTrfm}. We set the base point of $(X',R',B')$ as
  \begin{itemize}
    \item The base points are unchanged under (1), (2) and (3),
    \item Under (4), let $x$ be the base point of $xw^{-1}$.
  \end{itemize}
  If two presentations of a group $G$ is connected by a finite sequence of two transformations presented above or their inverse, we say that two presentations are {\it strongly Tietze equivalent preserving base point}.\hfill{$\square$}
\end{dfn}

Then, the following theorem holds.
\begin{thm}\label{Thm-StronglyTietzeEquiv-is-Group-Weighted-Equiv}
    Let $(X,R,B), (X',R',B')$ be two presentations of a group $G$ that are strongly Tietze equivalent preserving base points. Then, $\Gamma(X,R,B) \sim \Gamma(X',R',B')$. 
\end{thm}

\begin{proof}
    The graph is unchanged under (1) or (2) because we fix the base points, and (4) corresponds to source elimination or inverse operation. The nontrivial case arises when operation (3) is composed with other operations.
    \[
    \begin{aligned}
        (A) \quad r_{i} &\longleftrightarrow wr_{i}w^{-1}w'r_{j}w'^{-1} \\
        (B) \quad r_{i} &\longleftrightarrow wr_{i}w^{-1}w'r_{j}^{-1}w'^{-1}
    \end{aligned}
    \]
    We assume $r_{i} = x_{i}f_{i}^{-1}$, $r_{j} = x_{j}f_{j}^{-1}$, and since the base point of right-hand side is $x_{i}$,
    \[
    \begin{aligned}
        &(A) \quad wr_{i}w^{-1}w'r_{j}w'^{-1} \longrightarrow r_{i}w^{-1}w'r_{j}w'^{-1}w = x_{i}f_{i}^{-1}(w^{-1}w')r_{j}(w^{-1}w')^{-1}\\
        &(B) \quad wr_{i}w^{-1}w'r_{j}^{-1}w'^{-1} \longrightarrow r_{i}w^{-1}w'r_{j}^{-1}w'^{-1}w = x_{i}f_{i}^{-1}(w^{-1}w')r_{j}^{-1}(w^{-1}w')^{-1}.
    \end{aligned}
    \]
    Rewriting \( w^{-1}w' \) as \( w \), we obtain the following transformation:
    \[
    \begin{aligned}
        (A) \quad x_{i}f_{i}^{-1} &\longleftrightarrow x_{i}f_{i}^{-1}wr_{j}w^{-1}\\
        (B) \quad x_{i}f_{i}^{-1} &\longleftrightarrow x_{i}f_{i}^{-1}w r_{j}^{-1}w^{-1}.
    \end{aligned}
    \]
    As (A),
    \begin{align*}
        d(x_{i} f_{i}^{-1} w r_{j} w^{-1})
        &= dx_{i} - x_{i} f_{i}^{-1}df_{i} + x_{i} f_{i}^{-1} dw + x_{i} f_{i}^{-1} w dr_{j} - x_{i} f_{i}^{-1} w r_{j} w^{-1} dw\\
        &= dx_{i} - df_{i} + dw + w dr_{j} - dw\\
        &= dx_{i} - df_{i} + wdr_{j}.
    \end{align*}
    As (B), since $dr_{j}^{-1} = -r_{j}^{-1}dr_{j}$, and we can assume $r_{j}=1$ when constructing a group-weighted graph, we write
    \[
        d(x_{i} f_{i}^{-1} w r_{j}^{-1} w^{-1}) = dx_{i} - df_{i} - wdr_{j}.
    \]
    In addition, since $r_{j} = x_{j}f_{j}^{-1}$,
    \[
    \begin{aligned}
        &(A) \quad dx_{i}=df_{i} \longleftrightarrow dx_{i} = df_{i} - w(dx_{j} - df_{j})\\
        &(B) \quad dx_{i}=df_{i} \longleftrightarrow dx_{i} = df_{i} + w(dx_{j} - df_{j}).
    \end{aligned}
    \]
    Now, we write
    \begin{align*}
    df_{i} = \sum_{l} u_{l} dx_{l}, df_{j} = \sum_{l} w_{l} dx_{l}.
    \end{align*}

    Let us consider the case (A) first, the corresponding graph transformation is given below where $l\neq j$:
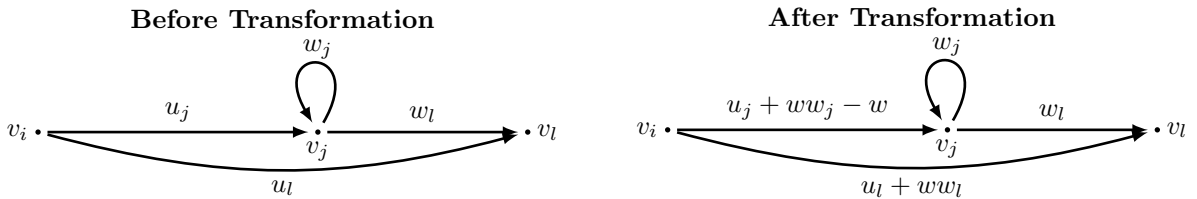
    \begin{center}
        \begin{minipage}{0.48\textwidth}
            \centering
            \textbf{Before Transformation}  
            \begin{tikzpicture}[x=0.92cm,y=0.92cm]
              \node (vi) at (-4,0) {};
              \node (vj) at (0,0) {};
              \node (vl) at (3,0) {};
              \node (weight) at (0,1.2) {$w_{j}$};
    
              \fill (vi) circle [radius=1pt];
              \fill (vj) circle [radius=1pt];
              \fill (vl) circle [radius=1pt];
    
              \draw (vi) node [left] {$v_{i}$};
              \draw (vj) node [below] {$v_{j}$};
              \draw (vl) node [right] {$v_{l}$};
    
              \draw[line width = 1pt, arrows={ - latex}] (vi) to node[above]{$u_{j}$} (vj);
              \draw[line width = 1pt, arrows={ - latex}] (vi) to [out=360-15, in=180+15] node[below]{$u_{l}$} (vl);
              \draw[line width = 1pt, arrows={ - latex}] (vj) to node[above]{$w_{l}$} (vl);
              \path[line width = 1pt, arrows={ - latex}] (vj) edge [out=90-30, in=90+30, distance=12mm] (vj);
    
            \end{tikzpicture}
        \end{minipage}
        \begin{minipage}{0.48\textwidth}
            \centering
            \textbf{After Transformation}  
            \begin{tikzpicture}[x=0.92cm,y=0.92cm]
              \node (vi) at (-4,0) {};
              \node (vj) at (0,0) {};
              \node (vl) at (3,0) {};
              \node (weight) at (0,1.2) {$w_{j}$};
    
              \fill (vi) circle [radius=1pt];
              \fill (vj) circle [radius=1pt];
              \fill (vl) circle [radius=1pt];
    
              \draw (vi) node [left] {$v_{i}$};
              \draw (vj) node [below] {$v_{j}$};
              \draw (vl) node [right] {$v_{l}$};
    
              \draw[line width = 1pt, arrows={ - latex}] (vi) to node[above]{$u_{j}+ww_{j}-w$} (vj);
              \draw[line width = 1pt, arrows={ - latex}] (vi) to [out=360-15, in=180+15] node[below]{$u_{l} + ww_{l}$} (vl);
              \draw[line width = 1pt, arrows={ - latex}] (vj) to node[above]{$w_{l}$} (vl);
              \path[line width = 1pt, arrows={ - latex}] (vj) edge [out=90-30, in=90+30, distance=12mm] (vj);
    
            \end{tikzpicture}
        \end{minipage}
    \end{center}

    These two need to be connected by a sequence of elementary transformations. We start with the left diagram. The graph is equivalent to the following figure by weight summation:
    
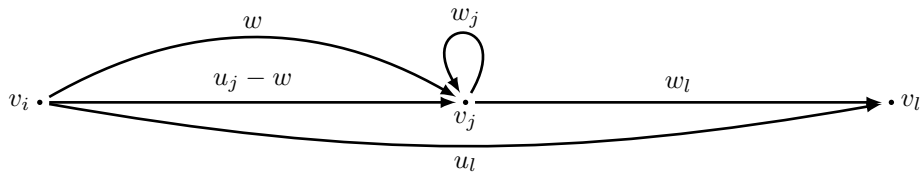
\begin{center}
        \begin{tikzpicture}[x=1.4cm,y=1.4cm]
            \node (vi) at (-4,0) {};
            \node (vj) at (0,0) {};
            \node (vl) at (4,0) {};
            \node (weight) at (0,0.8) {$w_{j}$};
            
            
            \fill (vi) circle [radius=1pt];
            \fill (vj) circle [radius=1pt];
            \fill (vl) circle [radius=1pt];
            
            \draw (vi) node [left] {$v_{i}$};
            \draw (vj) node [below] {$v_{j}$};
            \draw (vl) node [right] {$v_{l}$};
            
            \draw[line width = 1pt, arrows={ - latex}] (vi) to node[above]{$u_{j}-w$} (vj);
            \draw[line width = 1pt, arrows={ - latex}] (vi) to [out=30, in=150] node[above]{$w$} (vj);
            \draw[line width = 1pt, arrows={ - latex}] (vi) to [out=360-10, in=180+10] node[below]{$u_{l}$} (vl);
            \draw[line width = 1pt, arrows={ - latex}] (vj) to node[above]{$w_{l}$} (vl);
            \path[line width = 1pt, arrows={ - latex}] (vj) edge [out=90-30, in=90+30, distance=12mm] (vj);
        \end{tikzpicture}
    \end{center}
    
    Of the edges from $v_{i}$ to $v_{j}$, use the hub vertex resolution for an edge with weight $w$:
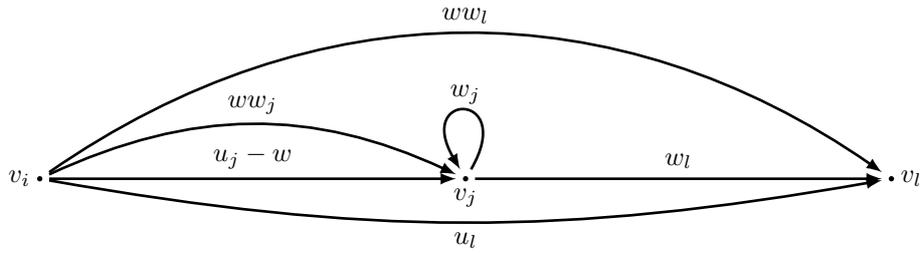
    \begin{center}
        \begin{tikzpicture}[x=1.4cm,y=1.4cm]
          \node (vi) at (-4,0) {};
          \node (vj) at (0,0) {};
          \node (vl) at (4,0) {};
          \node (weight) at (0,0.8) {$w_{j}$};
        
        
          \fill (vi) circle [radius=1pt];
          \fill (vj) circle [radius=1pt];
          \fill (vl) circle [radius=1pt];
        
          \draw (vi) node [left] {$v_{i}$};
          \draw (vj) node [below] {$v_{j}$};
          \draw (vl) node [right] {$v_{l}$};
        
          \draw[line width = 1pt, arrows={ - latex}] (vi) to node[above]{$u_{j}-w$} (vj);
          \draw[line width = 1pt, arrows={ - latex}] (vi) to [out=0+25, in=180-25] node[above]{$ww_{j}$} (vj);
          \draw[line width = 1pt, arrows={ - latex}] (vi) to [out=90-55, in=90+55] node[above]{$ww_{l}$} (vl);
          \draw[line width = 1pt, arrows={ - latex}] (vi) to [out=360-10, in=180+10] node[below]{$u_{l}$} (vl);
          \draw[line width = 1pt, arrows={ - latex}] (vj) to node[above]{$w_{l}$} (vl);
          \path[line width = 1pt, arrows={ - latex}] (vj) edge [out=90-30, in=90+30, distance=12mm] (vj);
        
        \end{tikzpicture}
    \end{center}
    After applying weight summation, we obtain the following graph:
    
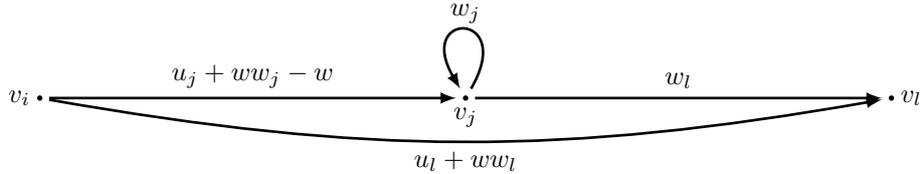
\begin{center}
        \begin{tikzpicture}[x=1.4cm,y=1.4cm]
          \node (vi) at (-4,0) {};
          \node (vj) at (0,0) {};
          \node (vl) at (4,0) {};
          \node (weight) at (0,0.8) {$w_{j}$};
        
        
          \fill (vi) circle [radius=1pt];
          \fill (vj) circle [radius=1pt];
          \fill (vl) circle [radius=1pt];
        
          \draw (vi) node [left] {$v_{i}$};
          \draw (vj) node [below] {$v_{j}$};
          \draw (vl) node [right] {$v_{l}$};
        
          \draw[line width = 1pt, arrows={ - latex}] (vi) to node[above]{$u_{j}+ww_{j}-w$} (vj);
          \draw[line width = 1pt, arrows={ - latex}] (vi) to [out=360-10, in=180+10] node[below]{$u_{l} + ww_{l}$} (vl);
          \draw[line width = 1pt, arrows={ - latex}] (vj) to node[above]{$w_{l}$} (vl);
          \path[line width = 1pt, arrows={ - latex}] (vj) edge [out=90-30, in=90+30, distance=12mm] (vj);
        
        \end{tikzpicture}
    \end{center}
    
    For the case (B), use weight summation $u_{j} = (u_{j}+w) + (-w)$ instead of $u_{j} = (u_{j}-w) + w$.
\end{proof}

Now, we can construct a sequence of transformations of group-weighted graphs corresponding to those of finitely presented groups. Is the ``inverse'' correct? In other words, can we construct a sequence of transformations of finitely presented groups corresponding to that of group-weighted graphs? The affirmative answer to this question is proved by using Theorem \ref{Thm-StronglyTietzeEquiv-is-Group-Weighted-Equiv}.

\begin{cor} \label{Thm-Group-Weighted-Equiv-And-Strong-Tietze}
  Let $(X,R,B)$ be a finitely presented group, and $\Gamma(X,R,B)$ be a group-weighted graph made by $(X,R,B)$. We construct a $\Gamma$ which is equivalent to $\Gamma(X,R,B)$ through a finite sequence of elementary transformations. Then there exists a finitely presented group that satisfies the following properties:
  \begin{itemize}
    \item $\Gamma = \Gamma(X', R', B')$.
    \item $(X,R)$ and $(X',R')$ are strongly Tietze equivalent.\hfill{$\square$}
  \end{itemize}
\end{cor}

\section{Application to knot theory}
\subsection{Invariance of Twisted Alexander polynomial from the view of holonomy preserving}
Let $K$ be an oriented knot and fix a diagram $D_{K}$ of the knot $K$. It is well known that two knots \( K \) and \( K' \) are equivalent if and only if their diagrams \( D_K \) and \( D_{K'} \) are connected by a finite sequence of the following Reidemeister moves (see, e.g.,~\cite{Otsuki-Quantum}).

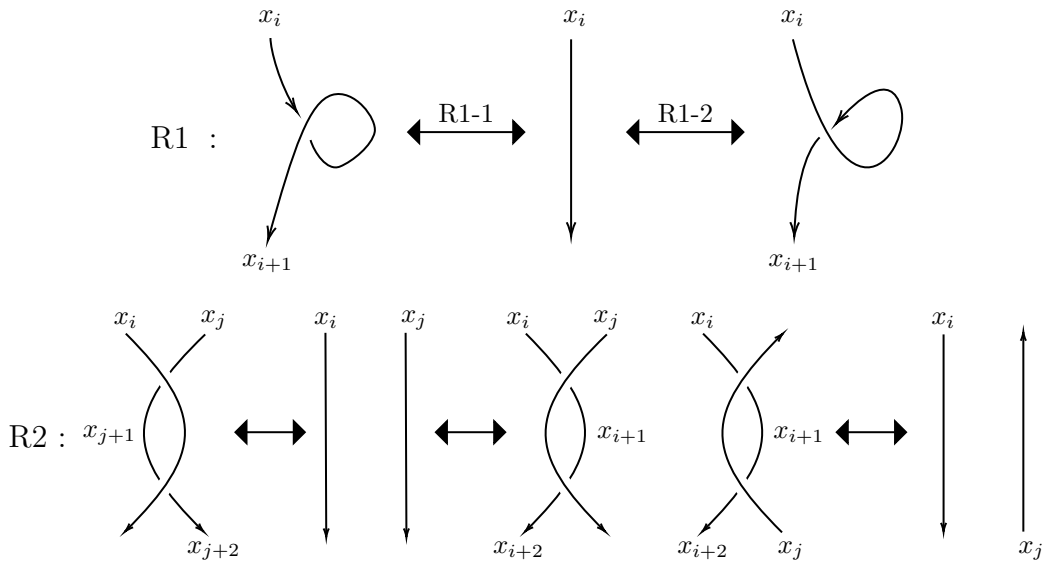
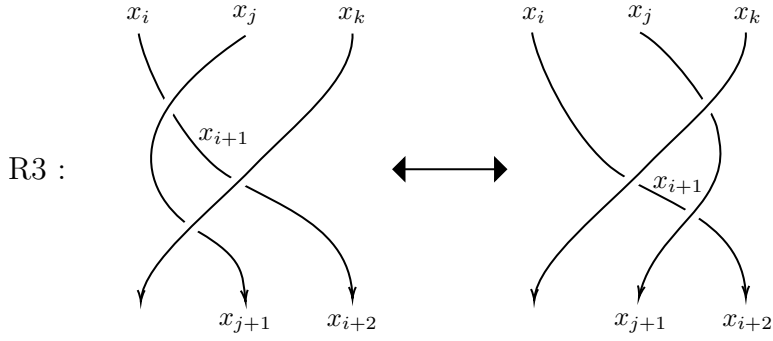
\begin{center}
  \tikzset{every picture/.style={line width=0.75pt}} 

  \begin{tikzpicture}[x=0.5pt,y=0.5pt,yscale=-1,xscale=1]

    \draw    (152,50) .. controls (152,58.78) and (158.65,84.66) .. (171.04,104.49) ;
    \draw [shift={(172,106)}, rotate = 236.98] [color={rgb, 255:red, 0; green, 0; blue, 0 }  ][line width=0.75]    (10.93,-3.29) .. controls (6.95,-1.4) and (3.31,-0.3) .. (0,0) .. controls (3.31,0.3) and (6.95,1.4) .. (10.93,3.29)   ;
    \draw    (182,127) .. controls (184.59,135.98) and (194.26,150.96) .. (204,147) .. controls (213.74,143.04) and (228,131) .. (230,121) .. controls (232,111) and (214.86,87.84) .. (198,93) .. controls (181.23,98.13) and (169.53,131.64) .. (150.29,201.94) ;
    \draw [shift={(150,203)}, rotate = 285.29] [color={rgb, 255:red, 0; green, 0; blue, 0 }  ][line width=0.75]    (10.93,-3.29) .. controls (6.95,-1.4) and (3.31,-0.3) .. (0,0) .. controls (3.31,0.3) and (6.95,1.4) .. (10.93,3.29)   ;
    \draw    (377,51) -- (377,198) ;
    \draw [shift={(377,200)}, rotate = 270] [color={rgb, 255:red, 0; green, 0; blue, 0 }  ][line width=0.75]    (10.93,-3.29) .. controls (6.95,-1.4) and (3.31,-0.3) .. (0,0) .. controls (3.31,0.3) and (6.95,1.4) .. (10.93,3.29)   ;
    \draw    (543,51) .. controls (583.65,208.6) and (631.49,133.04) .. (624,103) .. controls (616.69,73.71) and (590.69,96.58) .. (576.1,113.7) ;
    \draw [shift={(575,115)}, rotate = 309.81] [color={rgb, 255:red, 0; green, 0; blue, 0 }  ][line width=0.75]    (10.93,-3.29) .. controls (6.95,-1.4) and (3.31,-0.3) .. (0,0) .. controls (3.31,0.3) and (6.95,1.4) .. (10.93,3.29)   ;
    \draw    (563,125) .. controls (551.18,136.82) and (545.18,165.13) .. (544.05,198.47) ;
    \draw [shift={(544,200)}, rotate = 271.68] [color={rgb, 255:red, 0; green, 0; blue, 0 }  ][line width=0.75]    (10.93,-3.29) .. controls (6.95,-1.4) and (3.31,-0.3) .. (0,0) .. controls (3.31,0.3) and (6.95,1.4) .. (10.93,3.29)   ;
    \draw[arrows={triangle 90-triangle 90}, line width=0.75pt] (254,121) -- (343,121) 
    node[above, midway] {R1-1};

    \draw[arrows={triangle 90-triangle 90}, line width=0.75pt] (418,121) -- (507,121)
    node[above, midway] {R1-2};

    \node at (152,35) {\(x_i\)};
    \node at (150,220) {\(x_{i+1}\)};

    \node at (380,35) {\(x_{i}\)};

    \node at (543,35) {\(x_i\)};
    \node at (544,220) {\(x_{i+1}\)};

    \draw (60,114.4) node [anchor=north west][inner sep=0.75pt]  [font=\Large]  {$\mathrm{R} 1\ :\ $};

  \end{tikzpicture}
\end{center}

\begin{center}
  \tikzset{every picture/.style={line width=0.75pt}} 

  \begin{tikzpicture}[x=0.3pt,y=0.3pt,yscale=-1,xscale=1]

    \draw[arrows={triangle 90-triangle 90}, line width=0.75pt] (535,176) -- (625,176);
    \draw[arrows={triangle 90-triangle 90}, line width=0.75pt] (285,176) -- (375,176);
    \draw    (399,51) -- (399.99,306) ;
    \draw [shift={(400,308)}, rotate = 269.78] [color={rgb, 255:red, 0; green, 0; blue, 0 }  ][line width=0.75]    (10.93,-3.29) .. controls (6.95,-1.4) and (3.31,-0.3) .. (0,0) .. controls (3.31,0.3) and (6.95,1.4) .. (10.93,3.29)   ;
    \draw    (499,51) -- (499.99,306) ;
    \draw [shift={(500,308)}, rotate = 269.78] [color={rgb, 255:red, 0; green, 0; blue, 0 }  ][line width=0.75]    (10.93,-3.29) .. controls (6.95,-1.4) and (3.31,-0.3) .. (0,0) .. controls (3.31,0.3) and (6.95,1.4) .. (10.93,3.29)   ;
    \draw    (249,53) .. controls (148,151) and (147,202) .. (247,301) ;
    \draw [shift={(247,301)}, rotate = 224.71] [color={rgb, 255:red, 0; green, 0; blue, 0 }  ][line width=0.75]    (10.93,-3.29) .. controls (6.95,-1.4) and (3.31,-0.3) .. (0,0) .. controls (3.31,0.3) and (6.95,1.4) .. (10.93,3.29)   ;
    \draw [color={rgb, 255:red, 255; green, 255; blue, 255 }  ,draw opacity=1 ][line width=6]    (150,52) .. controls (247.02,150) and (248,199.98) .. (154.82,295.12) ;
    \draw [shift={(149,301)}, rotate = 315] [color={rgb, 255:red, 255; green, 255; blue, 255 }  ,draw opacity=1 ][line width=6]    (40.44,-12.17) .. controls (25.71,-5.16) and (12.23,-1.11) .. (0,0) .. controls (12.23,1.11) and (25.71,5.17) .. (40.44,12.17)   ;
    \draw    (150,52) .. controls (249,152) and (248,202) .. (149,301) ;
    \draw [shift={(149,301)}, rotate = 315] [color={rgb, 255:red, 0; green, 0; blue, 0 }  ][line width=0.75]    (10.93,-3.29) .. controls (6.95,-1.4) and (3.31,-0.3) .. (0,0) .. controls (3.31,0.3) and (6.95,1.4) .. (10.93,3.29)   ;
    \draw    (649,52) .. controls (748,152) and (747,202) .. (648,301) ;
    \draw [shift={(648,301)}, rotate = 315] [color={rgb, 255:red, 0; green, 0; blue, 0 }  ][line width=0.75]    (10.93,-3.29) .. controls (6.95,-1.4) and (3.31,-0.3) .. (0,0) .. controls (3.31,0.3) and (6.95,1.4) .. (10.93,3.29)   ;
    \draw [color={rgb, 255:red, 255; green, 255; blue, 255 }  ,draw opacity=1 ][line width=6]    (750,53) .. controls (651.02,149.04) and (648.08,199.94) .. (742.12,295.12) ;
    \draw [shift={(748,301)}, rotate = 224.71] [color={rgb, 255:red, 255; green, 255; blue, 255 }  ,draw opacity=1 ][line width=6]    (40.44,-12.17) .. controls (25.71,-5.16) and (12.23,-1.11) .. (0,0) .. controls (12.23,1.11) and (25.71,5.17) .. (40.44,12.17)   ;
    \draw    (750,53) .. controls (649,151) and (648,202) .. (748,301) ;
    \draw [shift={(748,301)}, rotate = 224.71] [color={rgb, 255:red, 0; green, 0; blue, 0 }  ][line width=0.75]    (10.93,-3.29) .. controls (6.95,-1.4) and (3.31,-0.3) .. (0,0) .. controls (3.31,0.3) and (6.95,1.4) .. (10.93,3.29)   ;
    \draw    (870,52) .. controls (969,152) and (968,202) .. (869,301) ;
    \draw [shift={(869,301)}, rotate = 315] [color={rgb, 255:red, 0; green, 0; blue, 0 }  ][line width=0.75]    (10.93,-3.29) .. controls (6.95,-1.4) and (3.31,-0.3) .. (0,0) .. controls (3.31,0.3) and (6.95,1.4) .. (10.93,3.29)   ;
    \draw [color={rgb, 255:red, 255; green, 255; blue, 255 }  ,draw opacity=1 ][line width=6]    (968.5,302) .. controls (870.99,204) and (869.52,154.02) .. (963.62,57.94) ;
    \draw [shift={(969.5,52)}, rotate = 135] [color={rgb, 255:red, 255; green, 255; blue, 255 }  ,draw opacity=1 ][line width=6]    (40.44,-12.17) .. controls (25.71,-5.16) and (12.23,-1.11) .. (0,0) .. controls (12.23,1.11) and (25.71,5.17) .. (40.44,12.17)   ;
    \draw    (968.5,302) .. controls (869,202) and (869.5,152) .. (969.5,52) ;
    \draw [shift={(969.5,52)}, rotate = 135] [color={rgb, 255:red, 0; green, 0; blue, 0 }  ][line width=0.75]    (10.93,-3.29) .. controls (6.95,-1.4) and (3.31,-0.3) .. (0,0) .. controls (3.31,0.3) and (6.95,1.4) .. (10.93,3.29)   ;

    \draw[arrows={triangle 90-triangle 90}, line width=0.75pt] (1035,176) -- (1125,176);

    \draw    (1170,53) -- (1170,299) ;  
    \draw [shift={(1170,301)}, rotate = 270] [color={rgb, 255:red, 0; green, 0; blue, 0 }  ][line width=0.75]    (10.93,-3.29) .. controls (6.95,-1.4) and (3.31,-0.3) .. (0,0) .. controls (3.31,0.3) and (6.95,1.4) .. (10.93,3.29)   ;
    
    \draw    (1269.5,301) -- (1269.5,52) ;  
    \draw [shift={(1269.5,50)}, rotate = 90] [color={rgb, 255:red, 0; green, 0; blue, 0 }  ][line width=0.75]    (10.93,-3.29) .. controls (6.95,-1.4) and (3.31,-0.3) .. (0,0) .. controls (3.31,0.3) and (6.95,1.4) .. (10.93,3.29)   ;

    \node at (150,35) {\(x_i\)};
    \node at (260,35) {\(x_j\)};

    \node at (130,180) {\(x_{j+1}\)};
    \node at (260,325) {\(x_{j+2}\)};

    \node at (400,35) {\(x_i\)};
    \node at (510,35) {\(x_{j}\)};

    \node at (640,35) {\(x_i\)};
    \node at (770,180) {\(x_{i+1}\)};
    \node at (640,325) {\(x_{i+2}\)};
    \node at (750,35) {\(x_{j}\)};

    \node at (870,35) {\(x_i\)};
    \node at (990,180) {\(x_{i+1}\)};
    \node at (870,325) {\(x_{i+2}\)};
    \node at (980,325) {\(x_{j}\)};

    \node at (1170,35) {\(x_i\)};  
    \node at (1280,325) {\(x_{j}\)};  

    \draw (0,164) node [anchor=north west][inner sep=0.75pt]  [font=\Large] [align=left] {R2 : };

  \end{tikzpicture}
\end{center}

\begin{center}

  \tikzset{every picture/.style={line width=0.75pt}} 

  \begin{tikzpicture}[x=0.4pt,y=0.4pt,yscale=-1,xscale=1]
    \draw (75,166) node [anchor=north west][inner sep=0.75pt]  [font=\Large] [align=left] {R3 : };

    \draw    (200,52) .. controls (204.54,80.06) and (233.48,142.76) .. (270,176) .. controls (306.34,209.07) and (398.08,223.85) .. (399.98,300.83) ;
    \draw [shift={(400,302)}, rotate = 269.27] [color={rgb, 255:red, 0; green, 0; blue, 0 }  ][line width=0.75]    (10.93,-3.29) .. controls (6.95,-1.4) and (3.31,-0.3) .. (0,0) .. controls (3.31,0.3) and (6.95,1.4) .. (10.93,3.29)   ;
    \draw [color={rgb, 255:red, 255; green, 255; blue, 255 }  ,draw opacity=1 ][line width=6]    (300,54) .. controls (166.87,143.04) and (212,213) .. (256,239) .. controls (294.5,261.75) and (297.02,276.08) .. (299.05,296.05) ;
    \draw [shift={(300,305)}, rotate = 263.16] [color={rgb, 255:red, 255; green, 255; blue, 255 }  ,draw opacity=1 ][line width=6]    (40.44,-12.17) .. controls (25.71,-5.16) and (12.23,-1.11) .. (0,0) .. controls (12.23,1.11) and (25.71,5.17) .. (40.44,12.17)   ;
    \draw    (300,54) .. controls (166.87,143.04) and (212,213) .. (256,239) .. controls (298.9,264.35) and (297.12,279.24) .. (299.79,303.14) ;
    \draw [shift={(300,305)}, rotate = 263.16] [color={rgb, 255:red, 0; green, 0; blue, 0 }  ][line width=0.75]    (10.93,-3.29) .. controls (6.95,-1.4) and (3.31,-0.3) .. (0,0) .. controls (3.31,0.3) and (6.95,1.4) .. (10.93,3.29)   ;
    \draw [color={rgb, 255:red, 255; green, 255; blue, 255 }  ,draw opacity=1 ][line width=6]    (400,52) .. controls (401.45,95.51) and (345.13,137.99) .. (307,178) .. controls (271.73,215.01) and (212.72,261.41) .. (203.28,296.66) ;
    \draw [shift={(202,305)}, rotate = 271.59] [color={rgb, 255:red, 255; green, 255; blue, 255 }  ,draw opacity=1 ][line width=6]    (40.44,-12.17) .. controls (25.71,-5.16) and (12.23,-1.11) .. (0,0) .. controls (12.23,1.11) and (25.71,5.17) .. (40.44,12.17)   ;
    \draw    (400,52) .. controls (401.45,95.51) and (345.13,137.99) .. (307,178) .. controls (269.44,217.41) and (204.97,267.47) .. (202.09,303.37) ;
    \draw [shift={(202,305)}, rotate = 271.59] [color={rgb, 255:red, 0; green, 0; blue, 0 }  ][line width=0.75]    (10.93,-3.29) .. controls (6.95,-1.4) and (3.31,-0.3) .. (0,0) .. controls (3.31,0.3) and (6.95,1.4) .. (10.93,3.29)   ;
    \draw    (568,51) .. controls (572.54,79.06) and (601.48,141.76) .. (638,175) .. controls (674.34,208.07) and (766.08,222.85) .. (767.98,299.83) ;
    \draw [shift={(768,301)}, rotate = 269.27] [color={rgb, 255:red, 0; green, 0; blue, 0 }  ][line width=0.75]    (10.93,-3.29) .. controls (6.95,-1.4) and (3.31,-0.3) .. (0,0) .. controls (3.31,0.3) and (6.95,1.4) .. (10.93,3.29)   ;
    \draw [color={rgb, 255:red, 255; green, 255; blue, 255 }  ,draw opacity=1 ][line width=6]    (669,51) .. controls (695.74,64.95) and (737.48,122.89) .. (741,141) .. controls (744.52,159.11) and (749,177) .. (731,205) .. controls (714.17,231.18) and (681.6,254.74) .. (670.08,292.84) ;
    \draw [shift={(668,301)}, rotate = 281.82] [color={rgb, 255:red, 255; green, 255; blue, 255 }  ,draw opacity=1 ][line width=6]    (40.44,-12.17) .. controls (25.71,-5.16) and (12.23,-1.11) .. (0,0) .. controls (12.23,1.11) and (25.71,5.17) .. (40.44,12.17)   ;
    \draw    (669,51) .. controls (695.74,64.95) and (737.48,122.89) .. (741,141) .. controls (744.52,159.11) and (749,177) .. (731,205) .. controls (713.27,232.58) and (678.08,257.25) .. (668.42,299.08) ;
    \draw [shift={(668,301)}, rotate = 281.82] [color={rgb, 255:red, 0; green, 0; blue, 0 }  ][line width=0.75]    (10.93,-3.29) .. controls (6.95,-1.4) and (3.31,-0.3) .. (0,0) .. controls (3.31,0.3) and (6.95,1.4) .. (10.93,3.29)   ;
    \draw [color={rgb, 255:red, 255; green, 255; blue, 255 }  ,draw opacity=1 ][line width=6]    (768,51) .. controls (769.45,94.51) and (713.13,136.99) .. (675,177) .. controls (639.73,214.01) and (580.72,260.41) .. (571.28,295.66) ;
    \draw [shift={(570,304)}, rotate = 271.59] [color={rgb, 255:red, 255; green, 255; blue, 255 }  ,draw opacity=1 ][line width=6]    (40.44,-12.17) .. controls (25.71,-5.16) and (12.23,-1.11) .. (0,0) .. controls (12.23,1.11) and (25.71,5.17) .. (40.44,12.17)   ;
    \draw    (768,51) .. controls (769.45,94.51) and (713.13,136.99) .. (675,177) .. controls (637.44,216.41) and (572.97,266.47) .. (570.09,302.37) ;
    \draw [shift={(570,304)}, rotate = 271.59] [color={rgb, 255:red, 0; green, 0; blue, 0 }  ][line width=0.75]    (10.93,-3.29) .. controls (6.95,-1.4) and (3.31,-0.3) .. (0,0) .. controls (3.31,0.3) and (6.95,1.4) .. (10.93,3.29)   ;
    \draw[arrows={triangle 90-triangle 90}, line width=0.75pt] (437,180) -- (545,180);

    \node at (200,35) {\(x_i\)};
    \node at (300,35) {\(x_j\)};
    \node at (400,35) {\(x_k\)};

    \node at (280,150) {\(x_{i+1}\)};
    \node at (300,325) {\(x_{j+1}\)};
    \node at (400,325) {\(x_{i+2}\)};

    \node at (570,35) {\(x_i\)};
    \node at (670,35) {\(x_{j}\)};
    \node at (770,35) {\(x_{k}\)};

    \node at (705,195) {\(x_{i+1}\)};
    \node at (670,325) {\(x_{j+1}\)};
    \node at (770,325) {\(x_{i+2}\)};
  \end{tikzpicture}
\end{center}

We label each arc in the knot diagram as shown in the figures above, ensuring that the indices increase sequentially as the strands pass under crossings. 
This labeling allows us to define the following relations for each crossing:

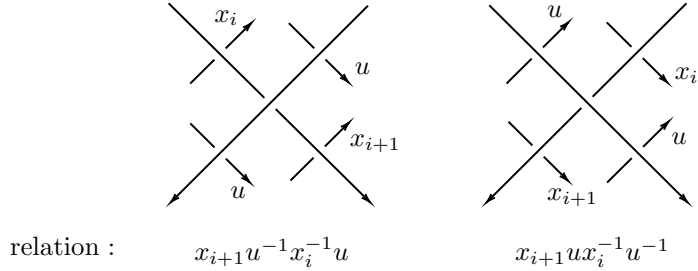
\begin{center}
  \tikzset{every picture/.style={line width=0.75pt}} 

  \begin{tikzpicture}[x=0.3pt,y=0.3pt,yscale=-1,xscale=1]

    \draw    (598,74) -- (672.09,148.09) ;
    \draw [shift={(673.5,149.5)}, rotate = 225] [color={rgb, 255:red, 0; green, 0; blue, 0 }  ][line width=0.75]    (10.93,-3.29) .. controls (6.95,-1.4) and (3.31,-0.3) .. (0,0) .. controls (3.31,0.3) and (6.95,1.4) .. (10.93,3.29)   ;
    \draw    (598.5,275.5) -- (673.59,200.41) ;
    \draw [shift={(675,199)}, rotate = 135] [color={rgb, 255:red, 0; green, 0; blue, 0 }  ][line width=0.75]    (10.93,-3.29) .. controls (6.95,-1.4) and (3.31,-0.3) .. (0,0) .. controls (3.31,0.3) and (6.95,1.4) .. (10.93,3.29)   ;
    \draw    (475.5,199) -- (548.09,271.59) ;
    \draw [shift={(549.5,273)}, rotate = 225] [color={rgb, 255:red, 0; green, 0; blue, 0 }  ][line width=0.75]    (10.93,-3.29) .. controls (6.95,-1.4) and (3.31,-0.3) .. (0,0) .. controls (3.31,0.3) and (6.95,1.4) .. (10.93,3.29)   ;
    \draw    (473,151) -- (548.59,75.41) ;
    \draw [shift={(550,74)}, rotate = 135] [color={rgb, 255:red, 0; green, 0; blue, 0 }  ][line width=0.75]    (10.93,-3.29) .. controls (6.95,-1.4) and (3.31,-0.3) .. (0,0) .. controls (3.31,0.3) and (6.95,1.4) .. (10.93,3.29)   ;
    \draw [color={rgb, 255:red, 255; green, 255; blue, 255 }  ,draw opacity=1 ][fill={rgb, 255:red, 255; green, 255; blue, 255 }  ,fill opacity=1 ][line width=6]    (697.5,50) -- (454.85,293.62) ;
    \draw [shift={(448.5,300)}, rotate = 314.89] [color={rgb, 255:red, 255; green, 255; blue, 255 }  ,draw opacity=1 ][line width=6]    (40.44,-12.17) .. controls (25.71,-5.16) and (12.23,-1.11) .. (0,0) .. controls (12.23,1.11) and (25.71,5.17) .. (40.44,12.17)   ;
    \draw    (697.5,50) -- (449.91,298.58) ;
    \draw [shift={(448.5,300)}, rotate = 314.89] [color={rgb, 255:red, 0; green, 0; blue, 0 }  ][line width=0.75]    (10.93,-3.29) .. controls (6.95,-1.4) and (3.31,-0.3) .. (0,0) .. controls (3.31,0.3) and (6.95,1.4) .. (10.93,3.29)   ;
    \draw [color={rgb, 255:red, 255; green, 255; blue, 255 }  ,draw opacity=1 ][line width=6]    (451.5,52.5) -- (692.14,293.63) ;
    \draw [shift={(698.5,300)}, rotate = 225.06] [color={rgb, 255:red, 255; green, 255; blue, 255 }  ,draw opacity=1 ][line width=6]    (40.44,-12.17) .. controls (25.71,-5.16) and (12.23,-1.11) .. (0,0) .. controls (12.23,1.11) and (25.71,5.17) .. (40.44,12.17)   ;
    \draw    (451.5,52.5) -- (697.09,298.58) ;
    \draw [shift={(698.5,300)}, rotate = 225.06] [color={rgb, 255:red, 0; green, 0; blue, 0 }  ][line width=0.75]    (10.93,-3.29) .. controls (6.95,-1.4) and (3.31,-0.3) .. (0,0) .. controls (3.31,0.3) and (6.95,1.4) .. (10.93,3.29)   ;
    \draw [color={rgb, 255:red, 255; green, 255; blue, 255 }  ,draw opacity=1 ][line width=6]    (304.5,49) -- (52.36,301.63) ;
    \draw [shift={(46,308)}, rotate = 314.94] [color={rgb, 255:red, 255; green, 255; blue, 255 }  ,draw opacity=1 ][line width=6]    (40.44,-12.17) .. controls (25.71,-5.16) and (12.23,-1.11) .. (0,0) .. controls (12.23,1.11) and (25.71,5.17) .. (40.44,12.17)   ;
    \draw    (78.5,151) -- (153.09,76.41) ;
    \draw [shift={(154.5,75)}, rotate = 135] [color={rgb, 255:red, 0; green, 0; blue, 0 }  ][line width=0.75]    (10.93,-3.29) .. controls (6.95,-1.4) and (3.31,-0.3) .. (0,0) .. controls (3.31,0.3) and (6.95,1.4) .. (10.93,3.29)   ;
    \draw    (203.5,273) -- (275.09,201.41) ;
    \draw [shift={(276.5,200)}, rotate = 135] [color={rgb, 255:red, 0; green, 0; blue, 0 }  ][line width=0.75]    (10.93,-3.29) .. controls (6.95,-1.4) and (3.31,-0.3) .. (0,0) .. controls (3.31,0.3) and (6.95,1.4) .. (10.93,3.29)   ;
    \draw    (78.5,201) -- (151.09,273.59) ;
    \draw [shift={(152.5,275)}, rotate = 225] [color={rgb, 255:red, 0; green, 0; blue, 0 }  ][line width=0.75]    (10.93,-3.29) .. controls (6.95,-1.4) and (3.31,-0.3) .. (0,0) .. controls (3.31,0.3) and (6.95,1.4) .. (10.93,3.29)   ;
    \draw    (201.5,75) -- (274.09,147.59) ;
    \draw [shift={(275.5,149)}, rotate = 225] [color={rgb, 255:red, 0; green, 0; blue, 0 }  ][line width=0.75]    (10.93,-3.29) .. controls (6.95,-1.4) and (3.31,-0.3) .. (0,0) .. controls (3.31,0.3) and (6.95,1.4) .. (10.93,3.29)   ;
    \draw [color={rgb, 255:red, 255; green, 255; blue, 255 }  ,draw opacity=1 ][line width=6]    (51.5,50) -- (296.12,293.65) ;
    \draw [shift={(302.5,300)}, rotate = 224.89] [color={rgb, 255:red, 255; green, 255; blue, 255 }  ,draw opacity=1 ][line width=6]    (40.44,-12.17) .. controls (25.71,-5.16) and (12.23,-1.11) .. (0,0) .. controls (12.23,1.11) and (25.71,5.17) .. (40.44,12.17)   ;
    \draw    (51.5,50) -- (301.08,298.59) ;
    \draw [shift={(302.5,300)}, rotate = 224.89] [color={rgb, 255:red, 0; green, 0; blue, 0 }  ][line width=0.75]    (10.93,-3.29) .. controls (6.95,-1.4) and (3.31,-0.3) .. (0,0) .. controls (3.31,0.3) and (6.95,1.4) .. (10.93,3.29)   ;
    \draw [color={rgb, 255:red, 255; green, 255; blue, 255 }  ,draw opacity=1 ][line width=6]    (304.5,49) -- (60.85,293.62) ;
    \draw [shift={(54.5,300)}, rotate = 314.89] [color={rgb, 255:red, 255; green, 255; blue, 255 }  ,draw opacity=1 ][line width=6]    (40.44,-12.17) .. controls (25.71,-5.16) and (12.23,-1.11) .. (0,0) .. controls (12.23,1.11) and (25.71,5.17) .. (40.44,12.17)   ;
    \draw    (300.25,53) -- (55.91,298.58) ;
    \draw [shift={(54.5,300)}, rotate = 314.85] [color={rgb, 255:red, 0; green, 0; blue, 0 }  ][line width=0.75]    (10.93,-3.29) .. controls (6.95,-1.4) and (3.31,-0.3) .. (0,0) .. controls (3.31,0.3) and (6.95,1.4) .. (10.93,3.29)   ;

    \draw (679,127.4) node [anchor=north west][inner sep=0.75pt]  [font=\normalsize]  {$x_{i}$};
    \draw (520,279.4) node [anchor=north west][inner sep=0.75pt]  [font=\normalsize]  {$x_{i+1}$};
    \draw (520,49.4) node [anchor=north west][inner sep=0.75pt]  [font=\normalsize]  {$u$};
    \draw (675,209.4) node [anchor=north west][inner sep=0.75pt]  [font=\normalsize]  {$u$};
    \draw (274,211.4) node [anchor=north west][inner sep=0.75pt]  [font=\normalsize]  {$x_{i+1}$};
    \draw (105,55) node [anchor=north west][inner sep=0.75pt]  [font=\normalsize]  {$x_{i}$};
    \draw (280,120.4) node [anchor=north west][inner sep=0.75pt]  [font=\normalsize]  {$u$};
    \draw (125,277.4) node [anchor=north west][inner sep=0.75pt]  [font=\normalsize]  {$u$};

    \draw (-150,340) node [anchor=north west][inner sep=0.75pt]  [font=\normalsize]  {relation : };

    \draw (80,340) node [anchor=north west][inner sep=0.75pt]  [font=\normalsize]  {$x_{i+1}u^{-1}x_{i}^{-1}u$};

    \draw (480,340) node [anchor=north west][inner sep=0.75pt]  [font=\normalsize]  {$x_{i+1}ux_{i}^{-1}u^{-1}$};

  \end{tikzpicture}
\end{center}
Then, we have a presentation of Wirtinger representation of $\pi_{K} = \pi_{1}(S^{3}\setminus K)$, which is $\Braket{x_{1}, \ldots, x_{n} | r_{1}, \ldots, r_{n}}$ where $r_{i} = x_{i+1} u^{\pm 1} x_{i}^{-1} u^{\mp 1}$. Since 
\[
    \prod_{i=1}^{n} r_{i}^{\pm 1} = 1,
\]
we can omit one relation and define $\Braket{X | R} := \Braket{x_{1}, \ldots, x_{n} | r_{1}, \ldots, r_{n-1}}$. Let $\phi : F_{n} \twoheadrightarrow \pi_{K}$ be the surjective homomorphism, $\alpha$ be the abelianization homomorphism
\[
  \alpha : \pi_{K} \to \pi_{K}/[\pi_{K},\pi_{K}] \cong \mathbb{Z} \cong \braket{t} ; x_{i} \mapsto t
\]
and $\rho $ be the homomorphism $\pi_{K} \to GL_{m}(\mathbb{C})$. The maps $\phi$, $\rho$ and $\alpha$ naturally induce ring homomorphisms $\tilde{\phi} : \mathbb{Z}F_{n} \to \mathbb{Z}\pi_{K}$, $\tilde{\rho} : \mathbb{Z}\pi_{K} \to \mathbb{Z}GL_{m}(\mathbb{C})$ and $\tilde{\alpha} : \mathbb{Z}\pi_{K} \to \mathbb{Z}[t^{\pm 1}]$. We define the composition map
\[
  \Phi : \mathbb{Z}F_{n}
  \overset{\widetilde{\phi}}{\longrightarrow}
  \mathbb{Z}\pi_{K}
  \overset{\widetilde{\rho} \otimes \widetilde{\alpha}}{\longrightarrow} M_{m}(\mathbb{C}[t^{\pm 1}])
\]
where $F_{n} = \braket{x_{1}, \ldots, x_{n}}$. Let us consider the $(n-1) \times n$ matrix $M$ whose $(i,j)$th entry is the $m \times m$ matrix
\[
    \Phi\left(  \frac{\partial r_{i}}{\partial x_{j}} \right) \in M_{m}(\mathbb{C}[t^{\pm 1}]).
\]
For $1 \leq l \leq n$, let us denote by $M_{l}$ the $(n-1) \times (n-1)$ matrix obtained from $M$ by excluding the $l$th column. The twisted Alexander invariant of $K$ associated with a representation $\rho : \pi_{K} \to GL_{m}(\mathbb{C})$ is defined to be the rational function
\[
    \Delta_{K,\rho}(t) = \frac{\det M_{j}}{\det (1-\Phi(x_{j}))}.
\]
Goda~\cite{H-Goda} proved\footnote{Goda~\cite{H-Goda} used knot graphs. A knot graph is constructed not by a finitely presented group but by a knot diagram with a base point.} $\zeta_{\rho} (X,R,B)^{-1} \stackrel{\cdot}{=} \det M_{1}$, where $B$ is set as $r_{i} : x_{i} = u^{\pm 1} x_{i+1} u^{\mp 1}$ and $\stackrel{\cdot}{=}$ means left-hand side and right-hand side are equal up to units in $\mathbb{C}[t^{\pm 1}]$. Since $\Phi(1 - x_{i}) = \Phi(1 - x_{j})$ for all $i,j \in \{1,\ldots, n\}$, we shall prove that $\det M_{j}$ is an invariant of the knot $K$ in the following by using weighted graphs. 

We consider only Reidemeister move 3 here\footnote{For Reidemeister move 1, if necessary, we may use Proposition~\ref{Prop-Base-Choice} to assume the relation takes the form $x_i = x_{i+1}$.}. This move induces a change in the Wirtinger presentation of the knot group.
\begin{center}
    \tikzset{every picture/.style={line width=0.75pt}} 
    
    \begin{tikzpicture}[x=0.45pt,y=0.45pt,yscale=-1,xscale=1]
        \draw (55,345) node [anchor=north west][inner sep=0.75pt]    {$\textup{(P) :} \begin{cases}
        x_{i} =x_{j} x_{i+1} x_{j}^{-1} & \\
        x_{i+1} =x_{k} x_{i+2} x_{k}^{-1} & \\
        x_{j} =x_{k} x_{j+1} x_{k}^{-1} &
        \end{cases}$};
        \draw (460,345) node [anchor=north west][inner sep=0.75pt]    {$(\textup{P}') :  \begin{cases}
        x_{i} =x_{k} x_{i+1} x_{k}^{-1} & \\
        x_{i+1} =x_{j+1} x_{i+2} x_{j+1}^{-1} & \\
        x_{j} =x_{k} x_{j+1} x_{k}^{-1} &
        \end{cases}$};
        
        \draw[arrows={triangle 90-triangle 90}] (350,415) -- (450,415);
    \end{tikzpicture}
\end{center}
The corresponding group-weighted graphs are described below:
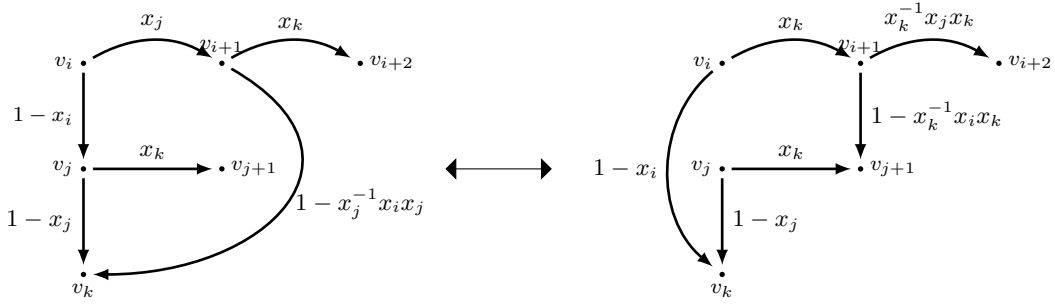
\begin{center}
    \begin{tikzpicture}[x=1.4cm,y=1.4cm]
        \node (vi) at (-1.3,1) {};
        \node (vi1) at (0,1) {};
        \node (vi2) at (1.3,1) {};
        \node (vj) at (-1.3, 0) {};
        \node (vj1) at (0,0) {};
        \node (vk) at (-1.3,-1) {};
        \draw (vi) node[left] {\footnotesize$v_{i}$};
        \draw (vi1) node[above] {\footnotesize$v_{i+1}$};
        \draw (vi2) node[right] {\footnotesize$v_{i+2}$};
        \draw (vj) node[left] {\footnotesize$v_{j}$};
        \draw (vj1) node[right] {\footnotesize$v_{j+1}$};
        \draw (vk) node[below] {\footnotesize$v_{k}$};
        
        \fill (vi) circle [radius=1pt];
        \fill (vi1) circle [radius=1pt];
        \fill (vi2) circle [radius=1pt];
        \fill (vj) circle [radius=1pt];
        \fill (vj1) circle [radius=1pt];
        \fill (vk) circle [radius=1pt];
        
        \draw[line width = 1pt, arrows={ - latex}] (vi) to [out=30, in=150] node[above] {\small$x_{j}$} (vi1);
        \draw[line width = 1pt, arrows={ - latex}] (vi) to node[left] {\footnotesize$1-x_{i}$} (vj);
        \draw[line width = 1pt, arrows={ - latex}] (vi1) to [out=30, in=150] node[above] {\small$x_{k}$} (vi2);
        \draw[line width = 1pt, arrows={ - latex}] (vi1) to [out=330, in=0, looseness = 1.8] node[right] {\small$1-x_{j}^{-1}x_{i}x_{j}$} (vk);
        \draw[line width = 1pt, arrows={ - latex}] (vj) to node[above] {\small$x_{k}$} (vj1);
        \draw[line width = 1pt, arrows={ - latex}] (vj) to node[left] {\small$1-x_{j}$} (vk);
        
        \draw[arrows={triangle 90-triangle 90}] (2.1,0) -- (3.1,0);
        
        \node (vi') at (4.7,1) {};
        \node (vi1') at (6,1) {};
        \node (vi2') at (7.3,1) {};
        \node (vj') at (4.7, 0) {};
        \node (vj1') at (6,0) {};
        \node (vk') at (4.7,-1) {};
        \draw (vi') node[left] {\footnotesize$v_{i}$};
        \draw (vi1') node[above] {\footnotesize$v_{i+1}$};
        \draw (vi2') node[right] {\footnotesize$v_{i+2}$};
        \draw (vj') node[left] {\footnotesize$v_{j}$};
        \draw (vj1') node[right] {\footnotesize$v_{j+1}$};
        \draw (vk') node[below] {\footnotesize$v_{k}$};
        
        \fill (vi') circle [radius=1pt];
        \fill (vi1') circle [radius=1pt];
        \fill (vi2') circle [radius=1pt];
        \fill (vj') circle [radius=1pt];
        \fill (vj1') circle [radius=1pt];
        \fill (vk') circle [radius=1pt];
        
        \draw[line width = 1pt, arrows={ - latex}] (vi') to [out=30, in=150] node[above] {\small$x_{k}$} (vi1');
        \draw[line width = 1pt, arrows={ - latex}] (vi') to [out=220, in=140, looseness = 1.0] node[left] {\small$1-x_{i}$} (vk');
        \draw[line width = 1pt, arrows={ - latex}] (vj') to node[above] {\small$x_{k}$} (vj1');
        \draw[line width = 1pt, arrows={ - latex}] (vj') to node[right] {\small$1-x_{j}$} (vk');
        \draw[line width = 1pt, arrows={ - latex}] (vi1') to [out=30, in=150] node[above] {\small$x_{k}^{-1}x_{j}x_{k}$} (vi2');
        \draw[line width = 1pt, arrows={ - latex}] (vi1') to node[right] {\small$1-x_{k}^{-1}x_{i}x_{k}$} (vj1');
    \end{tikzpicture}
\end{center}

Two group-weighted graphs are equivalent to the following group-weighted graph:
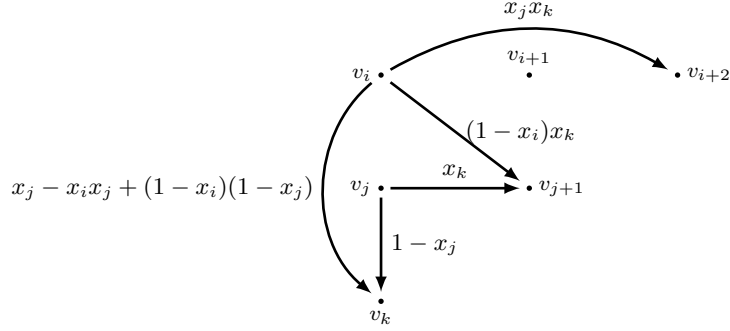
\begin{center}
\begin{tikzpicture}[x=1.5cm,y=1.5cm]
      \node (vi) at (-1.3,1) {};
      \node (vi1) at (0,1) {};
      \node (vi2) at (1.3,1) {};
      \node (vj) at (-1.3, 0) {};
      \node (vj1) at (0,0) {};
      \node (vk) at (-1.3,-1) {};
      \draw (vi) node[left] {\footnotesize$v_{i}$};
      \draw (vi1) node[above] {\footnotesize$v_{i+1}$};
      \draw (vi2) node[right] {\footnotesize$v_{i+2}$};
      \draw (vj) node[left] {\footnotesize$v_{j}$};
      \draw (vj1) node[right] {\footnotesize$v_{j+1}$};
      \draw (vk) node[below] {\footnotesize$v_{k}$};

     \fill (vi) circle [radius=1pt];
     \fill (vi1) circle [radius=1pt];
     \fill (vi2) circle [radius=1pt];
     \fill (vj) circle [radius=1pt];
     \fill (vj1) circle [radius=1pt];
     \fill (vk) circle [radius=1pt];

     \draw[line width = 1pt, arrows={ - latex}] (vi) to [out=30, in=150] node[above] {\small$x_{j}x_{k}$} (vi2);
     \draw[line width = 1pt, arrows={ - latex}] (vi) to node[right] {\small$(1-x_{i})x_{k}$} (vj1);
     \draw[line width = 1pt, arrows={ - latex}] (vi) to [out=220, in=140, looseness = 1.0] node[left] {\small$x_{j}-x_{i}x_{j} + (1-x_{i})(1-x_{j})$} (vk);
     \draw[line width = 1pt, arrows={ - latex}] (vj) to node[above] {\small$x_{k}$} (vj1);
     \draw[line width = 1pt, arrows={ - latex}] (vj) to node[right] {\small$1-x_{j}$} (vk);
    
    \end{tikzpicture}
\end{center}

Therefore, according to Proposition \ref{Prop-GrpEquivIsWtEquiv} and Theorem \ref{Thm-TrfmPreserveZeta}, $\Delta_{K,\rho}(t)$ is an invariant of the knot $K$. Furthermore, based on Theorem \ref{Thm-Group-Weighted-Equiv-And-Strong-Tietze}, there exists a representation which is strongly Tietze equivalent to \textup{(P)}. In fact, by carefully tracking the transformations used in the proof, we can identify this representation as \textup{(P}$'$\textup{)}.

As indicated by Proposition \ref{Prop-ElemProperty}, we now know elementary properties of the twisted Alexander polynomial:
\begin{itemize}
\item[{}] $\rho \sim \rho' \Longrightarrow \Delta_{K,\rho}(t) = \Delta_{K,\rho'}(t)$
\item[{}] $\rho \sim \rho_{1}\oplus \rho_{2} \Longrightarrow \Delta_{K,\rho}(t) = \Delta_{K,\rho_{1}}(t)\Delta_{K,\rho_{2}}(t)$.
\end{itemize}

\subsection{Quandles}
We have studied weights derived from group representations. To define knot invariants, it is not necessary to construct a detailed sequence of small graph deformations. 

Let us consider the group-weighted graphs made from knots before and after Reidemeister move 3. We denote sum of weights along all paths from $v_{i}$ to $v_{j}$ as $w(v_{i}\to v_{j})$. The key observation is that the following weights are preserved:
\begin{itemize}
    \item $w(v_{i}\to v_{i+2}) = x_{j}x_{k}$
    \item $w(v_{i}\to v_{j+1}) = x_{k}-x_{i}x_{k}$
    \item $w(v_{i}\to v_{k}) = 1-x_{j}$
    \item $w(v_{j}\to v_{j+1}) = x_{k}$
    \item $w(v_{j}\to v_{k}) = 1-x_{j}$
\end{itemize}

This preservation of weights guarantees that the holonomy structure of the knot graph remains invariant under this transformation. While we have so far described this phenomenon in terms of weighted paths, we will see later that this structure naturally leads to the notion of ``Alexander pairs'' and ``quandles'', providing a more algebraic interpretation of these preserved quantities.

More importantly, the essential property is that the graph holonomy remains invariant before and after deformation. This observation allows us to generalize weights on graphs from group presentations to quandle presentations. Under specific conditions, weights are determined as an ``Alexander pair,'' defined by Ishii and Oshiro~\cite{Ishii-Oshiro}, if and only if transformations of weights preserve holonomy as we will prove in Theore~\ref{Thm-PHiffAlexPair}.

The concept of a ``quandle'' was introduced by Joyce~\cite{Joyce} as an algebraic structure used to study knots. Quandles are closely related to ``Reidemeister moves'' and provide a natural framework for defining knot invariants.

\begin{dfn}
  Let $X$ be a non-empty set. If there exists a map $* : X\times X\to X$ satisfying the following conditions, then $(X, \ast)$ is called a \textit{quandle}:
  \begin{itemize}
    \item[(1)] For all $a \in X$, $a \ast a = a$.
    \item[(2)] For all $b \in X$, $*b : X\to X ; a\mapsto a*b$ is bijective.
    \item[(3)] For all $a, b, c \in X$, $(a\ast b)\ast c = (a\ast c)\ast (b\ast c)$.
    \hfill{$\square$}
  \end{itemize}
\end{dfn}

\begin{rem}
  Quandle can be regarded as a structure that algebraically handles knot crossings. Let us illustrate its meaning. Given a knot diagram, we assign elements of a set to each arc as follows:

  \begin{center}
    \tikzset{every picture/.style={line width=0.75pt}} 

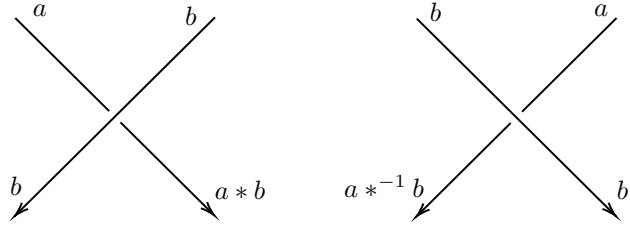
    \begin{tikzpicture}[x=0.5pt,y=0.5pt,yscale=-1,xscale=1]

      \draw    (501,51) -- (352.41,199.59) ;
      \draw [shift={(351,201)}, rotate = 315] [color={rgb, 255:red, 0; green, 0; blue, 0 }  ][line width=0.75]    (10.93,-3.29) .. controls (6.95,-1.4) and (3.31,-0.3) .. (0,0) .. controls (3.31,0.3) and (6.95,1.4) .. (10.93,3.29)   ;
      \draw    (51,51) -- (198.59,198.59) ;
      \draw [shift={(200,200)}, rotate = 225] [color={rgb, 255:red, 0; green, 0; blue, 0 }  ][line width=0.75]    (10.93,-3.29) .. controls (6.95,-1.4) and (3.31,-0.3) .. (0,0) .. controls (3.31,0.3) and (6.95,1.4) .. (10.93,3.29)   ;
      \draw [color={rgb, 255:red, 255; green, 255; blue, 255 }  ,draw opacity=1 ][line width=6]    (200.5,50.5) -- (56.86,194.14) ;
      \draw [shift={(50.5,200.5)}, rotate = 315] [color={rgb, 255:red, 255; green, 255; blue, 255 }  ,draw opacity=1 ][line width=6]    (40.44,-12.17) .. controls (25.71,-5.16) and (12.23,-1.11) .. (0,0) .. controls (12.23,1.11) and (25.71,5.17) .. (40.44,12.17)   ;
      \draw    (200.5,50.5) -- (51.91,199.09) ;
      \draw [shift={(50.5,200.5)}, rotate = 315] [color={rgb, 255:red, 0; green, 0; blue, 0 }  ][line width=0.75]    (10.93,-3.29) .. controls (6.95,-1.4) and (3.31,-0.3) .. (0,0) .. controls (3.31,0.3) and (6.95,1.4) .. (10.93,3.29)   ;
      \draw [color={rgb, 255:red, 255; green, 255; blue, 255 }  ,draw opacity=1 ][fill={rgb, 255:red, 0; green, 0; blue, 0 }  ,fill opacity=1 ][line width=6]    (351.5,51.5) -- (494.14,194.14) ;
      \draw [shift={(500.5,200.5)}, rotate = 225] [color={rgb, 255:red, 255; green, 255; blue, 255 }  ,draw opacity=1 ][line width=6]    (40.44,-12.17) .. controls (25.71,-5.16) and (12.23,-1.11) .. (0,0) .. controls (12.23,1.11) and (25.71,5.17) .. (40.44,12.17)   ;
      \draw    (351.5,51.5) -- (499.09,199.09) ;
      \draw [shift={(500.5,200.5)}, rotate = 225] [color={rgb, 255:red, 0; green, 0; blue, 0 }  ][line width=0.75]    (10.93,-3.29) .. controls (6.95,-1.4) and (3.31,-0.3) .. (0,0) .. controls (3.31,0.3) and (6.95,1.4) .. (10.93,3.29)   ;

      \draw (62,38.4) node [anchor=north west][inner sep=0.75pt]    {$a$};
      \draw (176,40.4) node [anchor=north west][inner sep=0.75pt]    {$b$};
      \draw (45,168.4) node [anchor=north west][inner sep=0.75pt]    {$b$};
      \draw (198,171.4) node [anchor=north west][inner sep=0.75pt]    {$a\ast b$};
      \draw (482,38.4) node [anchor=north west][inner sep=0.75pt]    {$a$};
      \draw (359,37.4) node [anchor=north west][inner sep=0.75pt]    {$b$};
      \draw (499,171.4) node [anchor=north west][inner sep=0.75pt]    {$b$};
      \draw (295,167.4) node [anchor=north west][inner sep=0.75pt]    {$a\ast ^{-1} b$};

    \end{tikzpicture}
  \end{center}
  Then, the algebraic conditions (1)-(3) that the operation $\ast^{\pm 1}$ must satisfy correspond to Reidemeister moves (1)-(3). For instance, condition (3) corresponds to Reidemeister move 3:
  \begin{center}
    \tikzset{every picture/.style={line width=0.75pt}} 

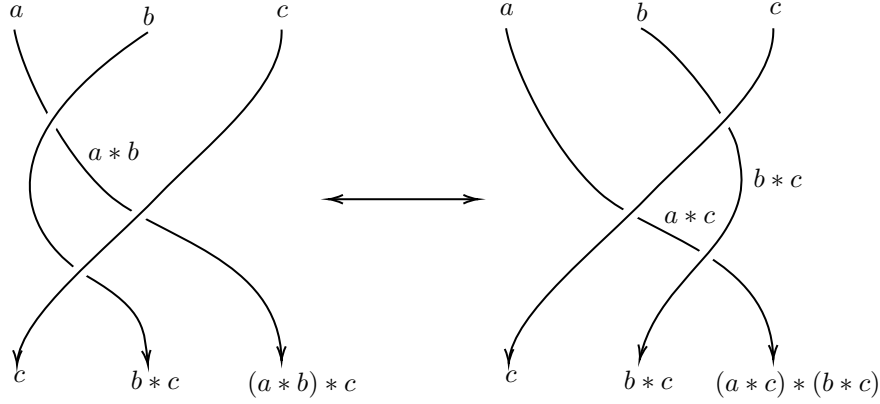
    \begin{tikzpicture}[x=0.5pt,y=0.5pt,yscale=-1,xscale=1]

        \draw    (51,52) .. controls (55.54,80.06) and (84.48,142.76) .. (121,176) .. controls (157.34,209.07) and (249.08,223.85) .. (250.98,300.83) ;
        \draw [shift={(251,302)}, rotate = 269.27] [color={rgb, 255:red, 0; green, 0; blue, 0 }  ][line width=0.75]    (10.93,-3.29) .. controls (6.95,-1.4) and (3.31,-0.3) .. (0,0) .. controls (3.31,0.3) and (6.95,1.4) .. (10.93,3.29)   ;
        \draw [color={rgb, 255:red, 255; green, 255; blue, 255 }  ,draw opacity=1 ][line width=6]    (151,54) .. controls (17.87,143.04) and (63,213) .. (107,239) .. controls (145.5,261.75) and (148.02,276.08) .. (150.05,296.05) ;
        \draw [shift={(151,305)}, rotate = 263.16] [color={rgb, 255:red, 255; green, 255; blue, 255 }  ,draw opacity=1 ][line width=6]    (40.44,-12.17) .. controls (25.71,-5.16) and (12.23,-1.11) .. (0,0) .. controls (12.23,1.11) and (25.71,5.17) .. (40.44,12.17)   ;
        \draw    (151,54) .. controls (17.87,143.04) and (63,213) .. (107,239) .. controls (149.9,264.35) and (148.12,279.24) .. (150.79,303.14) ;
        \draw [shift={(151,305)}, rotate = 263.16] [color={rgb, 255:red, 0; green, 0; blue, 0 }  ][line width=0.75]    (10.93,-3.29) .. controls (6.95,-1.4) and (3.31,-0.3) .. (0,0) .. controls (3.31,0.3) and (6.95,1.4) .. (10.93,3.29)   ;
        \draw [color={rgb, 255:red, 255; green, 255; blue, 255 }  ,draw opacity=1 ][line width=6]    (251,52) .. controls (252.45,95.51) and (196.13,137.99) .. (158,178) .. controls (122.73,215.01) and (63.72,261.41) .. (54.28,296.66) ;
        \draw [shift={(53,305)}, rotate = 271.59] [color={rgb, 255:red, 255; green, 255; blue, 255 }  ,draw opacity=1 ][line width=6]    (40.44,-12.17) .. controls (25.71,-5.16) and (12.23,-1.11) .. (0,0) .. controls (12.23,1.11) and (25.71,5.17) .. (40.44,12.17)   ;
        \draw    (251,52) .. controls (252.45,95.51) and (196.13,137.99) .. (158,178) .. controls (120.44,217.41) and (55.97,267.47) .. (53.09,303.37) ;
        \draw [shift={(53,305)}, rotate = 271.59] [color={rgb, 255:red, 0; green, 0; blue, 0 }  ][line width=0.75]    (10.93,-3.29) .. controls (6.95,-1.4) and (3.31,-0.3) .. (0,0) .. controls (3.31,0.3) and (6.95,1.4) .. (10.93,3.29)   ;
        \draw    (419,51) .. controls (423.54,79.06) and (452.48,141.76) .. (489,175) .. controls (525.34,208.07) and (617.08,222.85) .. (618.98,299.83) ;
        \draw [shift={(619,301)}, rotate = 269.27] [color={rgb, 255:red, 0; green, 0; blue, 0 }  ][line width=0.75]    (10.93,-3.29) .. controls (6.95,-1.4) and (3.31,-0.3) .. (0,0) .. controls (3.31,0.3) and (6.95,1.4) .. (10.93,3.29)   ;
        \draw [color={rgb, 255:red, 255; green, 255; blue, 255 }  ,draw opacity=1 ][line width=6]    (520,51) .. controls (546.74,64.95) and (588.48,122.89) .. (592,141) .. controls (595.52,159.11) and (600,177) .. (582,205) .. controls (565.17,231.18) and (532.6,254.74) .. (521.08,292.84) ;
        \draw [shift={(519,301)}, rotate = 281.82] [color={rgb, 255:red, 255; green, 255; blue, 255 }  ,draw opacity=1 ][line width=6]    (40.44,-12.17) .. controls (25.71,-5.16) and (12.23,-1.11) .. (0,0) .. controls (12.23,1.11) and (25.71,5.17) .. (40.44,12.17)   ;
        \draw    (520,51) .. controls (546.74,64.95) and (588.48,122.89) .. (592,141) .. controls (595.52,159.11) and (600,177) .. (582,205) .. controls (564.27,232.58) and (529.08,257.25) .. (519.42,299.08) ;
        \draw [shift={(519,301)}, rotate = 281.82] [color={rgb, 255:red, 0; green, 0; blue, 0 }  ][line width=0.75]    (10.93,-3.29) .. controls (6.95,-1.4) and (3.31,-0.3) .. (0,0) .. controls (3.31,0.3) and (6.95,1.4) .. (10.93,3.29)   ;
        \draw [color={rgb, 255:red, 255; green, 255; blue, 255 }  ,draw opacity=1 ][line width=6]    (619,51) .. controls (620.45,94.51) and (564.13,136.99) .. (526,177) .. controls (490.73,214.01) and (431.72,260.41) .. (422.28,295.66) ;
        \draw [shift={(421,304)}, rotate = 271.59] [color={rgb, 255:red, 255; green, 255; blue, 255 }  ,draw opacity=1 ][line width=6]    (40.44,-12.17) .. controls (25.71,-5.16) and (12.23,-1.11) .. (0,0) .. controls (12.23,1.11) and (25.71,5.17) .. (40.44,12.17)   ;
        \draw    (619,51) .. controls (620.45,94.51) and (564.13,136.99) .. (526,177) .. controls (488.44,216.41) and (423.97,266.47) .. (421.09,302.37) ;
        \draw [shift={(421,304)}, rotate = 271.59] [color={rgb, 255:red, 0; green, 0; blue, 0 }  ][line width=0.75]    (10.93,-3.29) .. controls (6.95,-1.4) and (3.31,-0.3) .. (0,0) .. controls (3.31,0.3) and (6.95,1.4) .. (10.93,3.29)   ;
        \draw [line width=0.75]    (288,180) -- (396,180) ;
        \draw [shift={(398,180)}, rotate = 180] [color={rgb, 255:red, 0; green, 0; blue, 0 }  ][line width=0.75]    (10.93,-3.29) .. controls (6.95,-1.4) and (3.31,-0.3) .. (0,0) .. controls (3.31,0.3) and (6.95,1.4) .. (10.93,3.29)   ;
        \draw [shift={(286,180)}, rotate = 0] [color={rgb, 255:red, 0; green, 0; blue, 0 }  ][line width=0.75]    (10.93,-3.29) .. controls (6.95,-1.4) and (3.31,-0.3) .. (0,0) .. controls (3.31,0.3) and (6.95,1.4) .. (10.93,3.29)   ;

        \draw (45,31.4) node [anchor=north west][inner sep=0.75pt]    {$a$};
        \draw (145,32.4) node [anchor=north west][inner sep=0.75pt]    {$b$};
        \draw (245,31.4) node [anchor=north west][inner sep=0.75pt]    {$c$};
        \draw (412,30.4) node [anchor=north west][inner sep=0.75pt]    {$a$};
        \draw (514,29.4) node [anchor=north west][inner sep=0.75pt]    {$b$};
        \draw (614,29.4) node [anchor=north west][inner sep=0.75pt]    {$c$};
        \draw (104,134.4) node [anchor=north west][inner sep=0.75pt]    {$a\ast b$};
        \draw (223,307.4) node [anchor=north west][inner sep=0.75pt]    {$( a\ast b) \ast c$};
        \draw (136,307.4) node [anchor=north west][inner sep=0.75pt]    {$b\ast c$};
        \draw (48,307.4) node [anchor=north west][inner sep=0.75pt]    {$c$};
        \draw (602,154.4) node [anchor=north west][inner sep=0.75pt]    {$b\ast c$};
        \draw (535,186.4) node [anchor=north west][inner sep=0.75pt]    {$a\ast c$};
        \draw (416,307.4) node [anchor=north west][inner sep=0.75pt]    {$c$};
        \draw (505,307.4) node [anchor=north west][inner sep=0.75pt]    {$b\ast c$};
        \draw (573,307.4) node [anchor=north west][inner sep=0.75pt]    {$( a\ast c) \ast ( b\ast c)$};

    \end{tikzpicture}

  \end{center}
  \hfill{$\square$}

\end{rem}

Unless otherwise ambiguous, we denote $(Q,\ast)$ simply as $Q$.

\subsection{Alexander pairs and holonomy preserving transformations}

\begin{dfn}[{\cite{Ishii-Oshiro}}]\label{Def-Alexander-pair}
    Let $Q$ be a quandle and $R$ be a unital ring. A pair of maps 
    $f_{1}, f_{2} : Q \times Q \to R$ is called an \textit{Alexander pair} 
    if it satisfies the following conditions:
    \begin{itemize}
        \item[(1)] For all $a \in Q$, $f_{1}(a,a) + f_{2}(a,a) = 1.$
        \item[(2)] For all $a,b \in Q$, $f_{1}(a,b)$ is an invertible element.
        \item[(3)] For all $a,b,c \in Q$, the following hold:
        \begin{itemize}
            \item[(a)] $f_{1}(a \ast b,c) f_{1}(a,b) = f_{1}(a \ast c, b\ast c)f_{1}(a,c).$
            \item[(b)] $f_{1}(a\ast b, c)f_{2}(a,b) = f_{2}(a\ast c, b\ast c)f_{1}(b,c).$
            \item[(c)] $f_{2}(a\ast b,c) = f_{1}(a\ast c, b\ast c) f_{2}(a,c) + f_{2}(a\ast c, b\ast c)f_{2}(b,c).$\hfill{$\square$}
        \end{itemize}
    \end{itemize}
\end{dfn}

\begin{prop}{\cite{ptHopf}} \label{prop Alex-pair-cond}
  Let $(Q, \ast)$ be a quandle and let $(f_{1}, f_{2}) : Q\times Q \to R$ be an Alexander pair. Then $(Q\times R, \star)$ forms a quandle with the following operation:
  \[
  (a,x) \star (b,y) = (a\ast b, f_{1}(a,b)x + f_{2}(a,b)y).
  \]
\end{prop}
\begin{proof}
    This follows directly from the definition.
\end{proof}

Now, we are going to consider ``what weights preserve holonomy'' before and after Reidemeister moves.
Firstly, we consider maps from \(Q\times Q\) to \(M_{n}(\mathbb{C}[t^{\pm 1}]) \). 
Our goal is to investigate the necessary conditions such maps must satisfy in order to define a well-posed knot invariant. 
To do so, we set maps \( g_{1}^{\pm 1}, g_{2}^{\pm 1} : Q\times Q \to M_{n}(\mathbb{C}[t^{\pm 1}]) \) as shown in the following figures.
Later, we will show that in order to define a well-defined knot invariant, \( g \) must satisfy the Alexander pair condition.

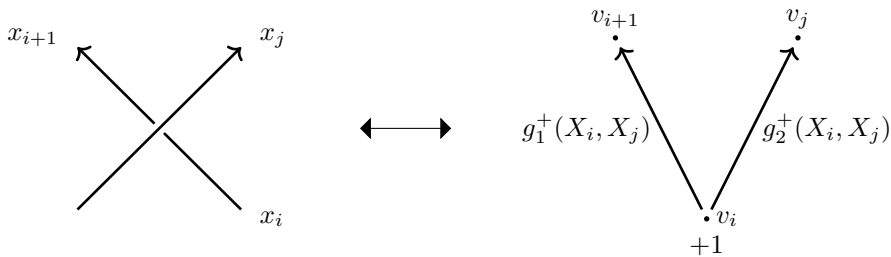
\begin{center}
  \begin{tikzpicture}[x=1.2cm,y=1.2cm]
      \node (xi1) at (-1,1) {};
      \node (xi) at (1,-1) {};
      \node (xj1) at (1,1) {};
      \node (xj) at (-1,-1) {};
      \draw (xi1) node [left] {$x_{i+1}$};
      \draw (xi) node [right] {$x_{i}$};
      \draw (xj1) node [right] {$x_{j}$};


      \draw[line width = 1pt, ->] (xi) -> (xi1);
      \draw[color=white, line width = 5pt] (xj) -- (xj1);
      \draw[line width = 1pt, ->] (xj) -> (xj1);

      \draw[arrows={triangle 90-triangle 90}] (2.2,0) -- (3.2,0);

       \node (A0) at (6,-1) {};
       \node (A1) at (5,1) {};
       \node (A2) at (7,1) {};
       \draw (A0) node [below = 3pt] {$+1$};
       \draw (A0) node [right] {$v_{i}$};
       \draw (A1) node [above] {$v_{i+1}$};
       \draw (A2) node [above] {$v_{j}$};

       \fill (A0) circle [radius=1pt];
       \fill (A1) circle [radius=1pt];
       \fill (A2) circle [radius=1pt];

       \draw[line width = 1pt, ->] (A0) to node[left]{$g^{+}_{1} (X_{i}, X_{j})$} (A1);
       \draw[line width = 1pt, ->] (A0) to node[right]{$g^{+}_{2} (X_{i}, X_{j})$} (A2);
  \end{tikzpicture}
\end{center}
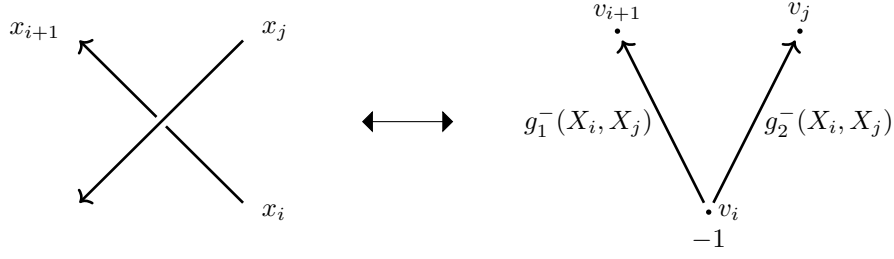
\begin{center}
  \begin{tikzpicture}[x=1.2cm,y=1.2cm]
      \node (xi1) at (-1,1) {};
      \node (xi) at (1,-1) {};
      \node (xj1) at (1,1) {};
      \node (xj) at (-1,-1) {};
      \draw (xi1) node [left] {$x_{i+1}$};
      \draw (xi) node [right] {$x_{i}$};
      \draw (xj1) node [right] {$x_{j}$};


      \draw[line width = 1pt, ->] (xi) -> (xi1);
      \draw[color=white, line width = 5pt] (xj) -- (xj1);
      \draw[line width = 1pt, ->] (xj1) -> (xj);

      \draw[arrows={triangle 90-triangle 90}] (2.2,0) -- (3.2,0);

      \node (A0) at (6,-1) {};
      \node (A1) at (5,1) {};
      \node (A2) at (7,1) {};
      \draw (A0) node [below = 3pt] {$-1$};
      \draw (A0) node [right] {$v_{i}$};
      \draw (A1) node [above] {$v_{i+1}$};
      \draw (A2) node [above] {$v_{j}$};

      \fill (A0) circle [radius=1pt];
      \fill (A1) circle [radius=1pt];
      \fill (A2) circle [radius=1pt];

      \draw[line width = 1pt, ->] (A0) to node[left]{$g^{-}_{1} (X_{i}, X_{j})$} (A1);
      \draw[line width = 1pt, ->] (A0) to node[right]{$g^{-}_{2} (X_{i}, X_{j})$} (A2);
  \end{tikzpicture}
\end{center}

Definition \ref{Def-Alexander-pair} shows the conditions the Alexander pair must satisfy. 
To define a knot invariant, it is necessary to ensure that these conditions hold for all pairs \( (X_i, X_j) \in Q \times Q \) that can be realized by varying \( Q(K) \) and \( \rho \). Thus, we define the following notion.
\begin{dfn}
    For knot graphs, $f = (f_{1}, f_{2}) : Q \times Q \to \mathbb{C}[t^{\pm 1}]$ satisfying the Alexander pair condition weight for ``all realizable pairs,'' i.e., any pair $(X_i, X_j) \in Q \times Q$ that can be realized by varying $Q(K)$ and $\rho$, is called \textit{$f$-twisted weight}.


    
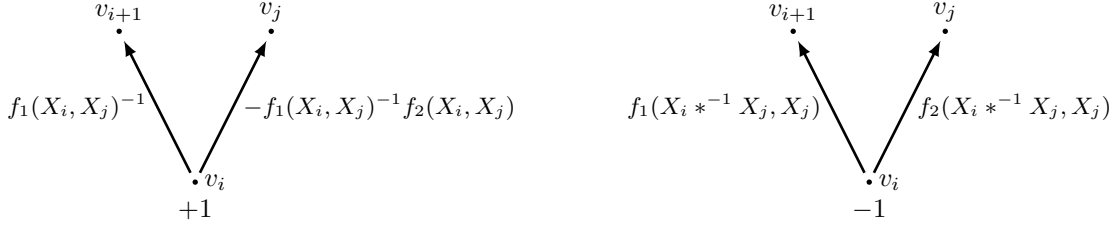
\begin{center}
        \begin{tikzpicture}[x=1.0cm,y=1.0cm]
            \node (A0) at (0,-1) {};
            \node (A1) at (-1,1) {};
            \node (A2) at (1,1) {};
            \draw (A0) node [below = 3pt] {$+1$};
            \draw (A0) node [right] {$v_{i}$};
            \draw (A1) node [above] {$v_{i+1}$};
            \draw (A2) node [above] {$v_{j}$};
            
            \fill (A0) circle [radius=1pt];
            \fill (A1) circle [radius=1pt];
            \fill (A2) circle [radius=1pt];
            
            \draw[line width = 1pt, arrows={ - latex}] (A0) to node[left] {\small$f_{1} (X_{i}, X_{j})^{-1}$} (A1);
            \draw[line width = 1pt, arrows={ - latex}] (A0) to node[right] {\small$-f_{1} (X_{i}, X_{j})^{-1}f_{2}(X_{i},X_{j})$} (A2);
        \end{tikzpicture}
        \hspace{10mm} 
        \begin{tikzpicture}[x=1.0cm,y=1.0cm]
            \node (A0) at (0,-1) {};
            \node (A1) at (-1,1) {};
            \node (A2) at (1,1) {};
            \draw (A0) node [below = 3pt] {$-1$};
            \draw (A0) node [right] {$v_{i}$};
            \draw (A1) node [above] {$v_{i+1}$};
            \draw (A2) node [above] {$v_{j}$};
            
            \fill (A0) circle [radius=1pt];
            \fill (A1) circle [radius=1pt];
            \fill (A2) circle [radius=1pt];
            
            \draw[line width = 1pt, arrows={ - latex}] (A0) to node[left] {\small$f_{1} (X_{i}\ast^{-1}X_{j}, X_{j})$} (A1);
            \draw[line width = 1pt, arrows={ - latex}] (A0) to node[right] {\small$f_{2} (X_{i}\ast^{-1}X_{j}, X_{j})$} (A2);
        \end{tikzpicture}
    \end{center}

    \vspace{-30pt}
    \begin{flushright}
        $\square$
    \end{flushright}
\end{dfn}
We describe the graph transformations corresponding to the Reidemeister moves. Note that for R1-2, a loop is included in the graph, so we exclude this case from our consideration for now. For clarity, we separate the description of Reidemeister move 2 into distinct cases and label the corresponding diagrams as (A) through (C).

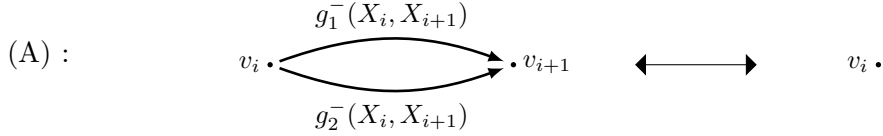
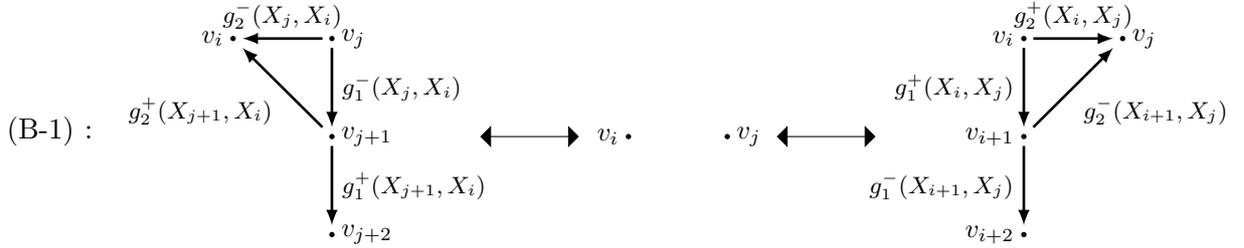
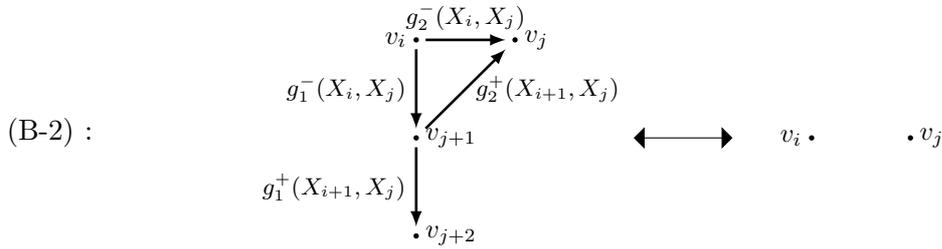
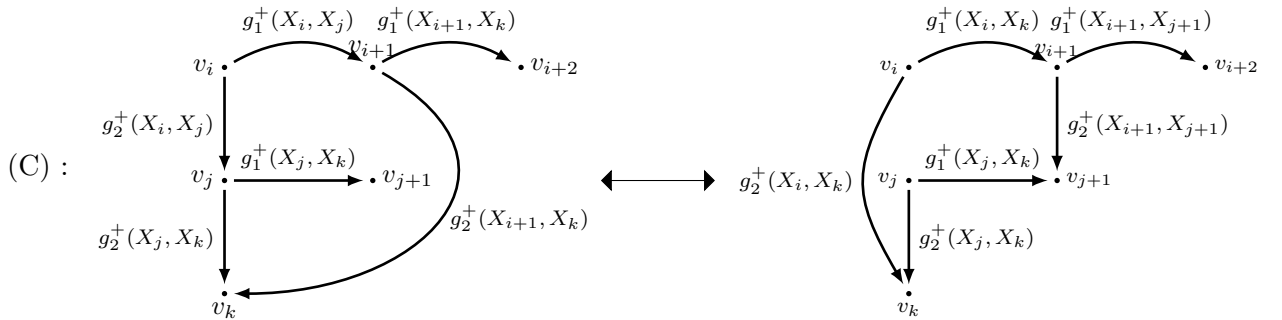
\begin{tikzpicture}
    \node (a) at (0,1) {\large (A) : };
    \node (empty) at (0,0) {};
\end{tikzpicture}%
\ \hspace{50pt}
\begin{tikzpicture}[x=1.6cm,y=1.6cm]
    \node (vi) at (-1,0) {};
    \node (vi1) at (1,0) {};
    \draw (vi) node [left] {$v_{i}$};
    \draw (vi1) node [right] {$v_{i+1}$};
    
    \fill (vi) circle [radius=1pt];
    \fill (vi1) circle [radius=1pt];
    
    \draw[line width = 1pt, arrows={ - latex}] (vi) to [out=20, in=160] node[above] {$g^{-}_{1}(X_{i}, X_{i+1})$} (vi1);
    \draw[line width = 1pt, arrows={ - latex}] (vi) to [out=-20, in=-160] node[below] {$g^{-}_{2}(X_{i}, X_{i+1})$} (vi1);
    
    \draw[arrows={triangle 90-triangle 90}] (2,0) -- (3,0);
    
    \coordinate (vi') at (4,0) node at (vi') [left] {$v_{i}$};
    
    \fill (vi') circle [radius=1pt];
\end{tikzpicture}\\

\begin{tikzpicture}
  \node (a) at (0,1.5) {\large (B-1) : };
  \node (empty) at (0,0) {};
\end{tikzpicture}%
\ \
\begin{tikzpicture}[x=1.3cm,y=1.3cm]
 \node (vi) at (-1,1) {};
 \node (vj) at (-0,1) {};
 \node (vj1) at (0,0) {};
 \node (vj2) at (0,-1) {};
 \draw (vi) node [left] {$v_{i}$};
 \draw (vj) node [right] {$v_{j}$};
 \draw (vj1) node [right] {$v_{j+1}$};
 \draw (vj2) node [right] {$v_{j+2}$};

 \fill (vi) circle [radius=1pt];
 \fill (vj) circle [radius=1pt];
 \fill (vj1) circle [radius=1pt];
 \fill (vj2) circle [radius=1pt];

 \draw[line width = 1pt, arrows={ - latex}] (vj) to node[above] {\small$g^{-}_{2}(X_{j}, X_{i})$} (vi);
 \draw[line width = 1pt, arrows={ - latex}] (vj) to node[right] {\small$g^{-}_{1}(X_{j}, X_{i})$} (vj1);
 \draw[line width = 1pt, arrows={ - latex}] (vj1) to  node[right] {\small$g^{+}_{1}(X_{j+1}, X_{i})$} (vj2);
 \draw[line width = 1pt, arrows={ - latex}] (vj1) to node[below left] {\small$g^{+}_{2}(X_{j+1}, X_{i})$} (vi);


 \draw[arrows={triangle 90-triangle 90}] (1.5,0) -- (2.5,0);

 \coordinate (vi') at (3,0) node at (vi') [left] {$v_{i}$};
 \coordinate (vj') at (4,0) node at (vj') [right] {$v_{j}$};

 \fill (vi') circle [radius=1pt];
 \fill (vj') circle [radius=1pt];

 \draw[arrows={triangle 90-triangle 90}] (4.5,0) -- (5.5,0);

 \node (vi'') at (7,1) {};
 \node (vi1'') at (7,0) {};
 \node (vi2'') at (7,-1) {};
 \node (vj'') at (8,1) {};
 \draw (vi'') node[left] {$v_{i}$};
 \draw (vi1'') node[left] {$v_{i+1}$};
 \draw (vi2'') node[left] {$v_{i+2}$};
 \draw (vj'') node[right] {$v_{j}$};

 \fill (vi'') circle [radius=1pt];
 \fill (vj'') circle [radius=1pt];
 \fill (vi1'') circle [radius=1pt];
 \fill (vi2'') circle [radius=1pt];

 \draw[line width = 1pt, arrows={ - latex}] (vi'') to node[above] {\small$g^{+}_{2}(X_{i}, X_{j})$} (vj'');
 \draw[line width = 1pt, arrows={ - latex}] (vi'') to  node[left] {\small$g^{+}_{1}(X_{i}, X_{j})$} (vi1'');
 \draw[line width = 1pt, arrows={ - latex}] (vi1'') to node[below right] {\small$g^{-}_{2}(X_{i+1}, X_{j})$} (vj'');
 \draw[line width = 1pt, arrows={ - latex}] (vi1'') to node[left] {\small$g^{-}_{1}(X_{i+1}, X_{j})$} (vi2'');


\end{tikzpicture}\\

\begin{tikzpicture}
  \node (a) at (0,1.5) {\large (B-2) : };
  \node (empty) at (0,0) {};
\end{tikzpicture}%
\ \hspace{50pt}
\begin{tikzpicture}[x=1.3cm,y=1.3cm]
 \node (vi) at (-1,1) {};
 \node (vj) at (-0,1) {};
 \node (vi1) at (-1,0) {};
 \node (vi2) at (-1,-1) {};
 \draw (vi) node [left] {$v_{i}$};
 \draw (vj) node [right] {$v_{j}$};
 \draw (vi1) node [right] {$v_{j+1}$};
 \draw (vi2) node [right] {$v_{j+2}$};

 \fill (vi) circle [radius=1pt];
 \fill (vj) circle [radius=1pt];
 \fill (vi1) circle [radius=1pt];
 \fill (vi2) circle [radius=1pt];

 \draw[line width = 1pt, arrows={ - latex}] (vi) to node[above] {\small$g^{-}_{2}(X_{i}, X_{j})$} (vj);
 \draw[line width = 1pt, arrows={ - latex}] (vi) to node[left] {\small$g^{-}_{1}(X_{i}, X_{j})$} (vi1);
 \draw[line width = 1pt, arrows={ - latex}] (vi1) to  node[right] {\small$g^{+}_{2}(X_{i+1}, X_{j})$} (vj);
 \draw[line width = 1pt, arrows={ - latex}] (vi1) to node[left] {\small$g^{+}_{1}(X_{i+1}, X_{j})$} (vi2);


 \draw[arrows={triangle 90-triangle 90}] (1.2,0) -- (2.2,0);

 \coordinate (vi') at (3,0) node at (vi') [left] {$v_{i}$};
 \coordinate (vj') at (4,0) node at (vj') [right] {$v_{j}$};

 \fill (vi') circle [radius=1pt];
 \fill (vj') circle [radius=1pt];

\end{tikzpicture}\\

\begin{tikzpicture}
  \node (a) at (0,2) {\large (C) : };
  \node (empty) at (0,0) {};
\end{tikzpicture}%
\ \
\begin{tikzpicture}[x=1.5cm,y=1.5cm]
 \node (vi) at (-1.3,1) {};
 \node (vi1) at (0,1) {};
 \node (vi2) at (1.3,1) {};
 \node (vj) at (-1.3, 0) {};
 \node (vj1) at (0,0) {};
 \node (vk) at (-1.3,-1) {};
 \draw (vi) node[left] {$v_{i}$};
 \draw (vi1) node[above] {$v_{i+1}$};
 \draw (vi2) node[right] {$v_{i+2}$};
 \draw (vj) node[left] {$v_{j}$};
 \draw (vj1) node[right] {$v_{j+1}$};
 \draw (vk) node[below] {$v_{k}$};

 \fill (vi) circle [radius=1pt];
 \fill (vi1) circle [radius=1pt];
 \fill (vi2) circle [radius=1pt];
 \fill (vj) circle [radius=1pt];
 \fill (vj1) circle [radius=1pt];
 \fill (vk) circle [radius=1pt];

 \draw[line width = 1pt, arrows={ - latex}] (vi) to [out=30, in=150] node[above] {\footnotesize$g^{+}_{1}(X_{i}, X_{j})$} (vi1);
 \draw[line width = 1pt, arrows={ - latex}] (vi) to node[left] {\footnotesize$g^{+}_{2}(X_{i}, X_{j})$} (vj);
 \draw[line width = 1pt, arrows={ - latex}] (vi1) to [out=30, in=150] node[above] {\footnotesize$g^{+}_{1}(X_{i+1}, X_{k})$} (vi2);
 \draw[line width = 1pt, arrows={ - latex}] (vi1) to [out=330, in=0, looseness = 1.8] node[right] {\footnotesize$g^{+}_{2}(X_{i+1}, X_{k})$} (vk);
 \draw[line width = 1pt, arrows={ - latex}] (vj) to node[above] {\footnotesize$g^{+}_{1}(X_{j}, X_{k})$} (vj1);
 \draw[line width = 1pt, arrows={ - latex}] (vj) to node[left] {\footnotesize$g^{+}_{2}(X_{j}, X_{k})$} (vk);

 \draw[arrows={triangle 90-triangle 90}] (2,0) -- (3,0);

 \node (vi') at (4.7,1) {};
 \node (vi1') at (6,1) {};
 \node (vi2') at (7.3,1) {};
 \node (vj') at (4.7, 0) {};
 \node (vj1') at (6,0) {};
 \node (vk') at (4.7,-1) {};
 \draw (vi') node[left] {\footnotesize$v_{i}$};
 \draw (vi1') node[above] {\footnotesize$v_{i+1}$};
 \draw (vi2') node[right] {\footnotesize$v_{i+2}$};
 \draw (vj') node[left] {\footnotesize$v_{j}$};
 \draw (vj1') node[right] {\footnotesize$v_{j+1}$};
 \draw (vk') node[below] {\footnotesize$v_{k}$};

 \fill (vi') circle [radius=1pt];
 \fill (vi1') circle [radius=1pt];
 \fill (vi2') circle [radius=1pt];
 \fill (vj') circle [radius=1pt];
 \fill (vj1') circle [radius=1pt];
 \fill (vk') circle [radius=1pt];

 \draw[line width = 1pt, arrows={ - latex}] (vi') to [out=30, in=150] node[above] {\footnotesize$g^{+}_{1}(X_{i}, X_{k})$} (vi1');
 \draw[line width = 1pt, arrows={ - latex}] (vi') to [out=240, in=120, looseness = 1.3] node[left] {\footnotesize$g^{+}_{2}(X_{i}, X_{k})$} (vk');
 \draw[line width = 1pt, arrows={ - latex}] (vi1') to [out=30, in=150] node[above] {\footnotesize$g^{+}_{1}(X_{i+1}, X_{j+1})$} (vi2');
 \draw[line width = 1pt, arrows={ - latex}] (vi1') to node[right] {\footnotesize$g^{+}_{2}(X_{i+1}, X_{j+1})$} (vj1');
 \draw[line width = 1pt, arrows={ - latex}] (vj') to node[above] {\footnotesize$g^{+}_{1}(X_{j}, X_{k})$} (vj1');
 \draw[line width = 1pt, arrows={ - latex}] (vj') to node[right] {\footnotesize$g^{+}_{2}(X_{j}, X_{k})$} (vk');

\end{tikzpicture}

To preserve holonomy, the followings have to be held\footnote{For (C), the end points of edges from $v_{i}$ are only $v_{i+1}, v_{k}$. So $w(v_{i}\to v_{i+1})$ need not to be preserved.}:
\[
\begin{aligned}
    (A) \quad &g_{1}^{-}(X_{i}, X_{i+1})+g_{2}^{-}(X_{i}, X_{i+1})=1.\\
    (B-1) \quad &g_{2}^{-}(X_{j}, X_{i})+g_{1}^{-}(X_{j}, X_{i})g_{2}^{+}(X_{j+1}, X_{i}) = g_{2}^{+}(X_{i}, X_{j})+g_{1}^{+}(X_{i}, X_{j})g_{2}^{-}(X_{i+1}, X_{i})= 0.\\
    \quad &g_{1}^{-}(X_{j}, X_{i})g_{1}^{+}(X_{j+1}, X_{i}) = g_{1}^{+}(X_{i}, X_{j})g_{1}^{-}(X_{i+1}, X_{j}) = 1\\
    (B-2) \quad &g_{2}^{-}(X_{i}, X_{j})+g_{1}^{-}(X_{i}, X_{j})g_{2}^{+}(X_{i+1}, X_{j}) = 0\\
    \quad &g_{1}^{-}(X_{i}, X_{j})g_{1}^{+}(X_{i+1}, X_{j}) =1\\
    (C) \quad &g_{1}^{+}(X_{i}, X_{j})g_{1}^{+}(X_{i+1}, X_{k}) = g_{1}^{+}(X_{i}, X_{k})g_{1}^{+}(X_{i+1}, X_{j+1})\\
    \quad &g_{2}^{+}(X_{i}, X_{j})g_{1}^{+}(X_{j}, X_{k})=g_{1}^{+}(X_{i}, X_{k})g_{2}^{+}(X_{i+1}, X_{j+1})\\
    \quad &g_{1}^{+}(X_{i}, X_{j})g_{2}^{+}(X_{i+1}, X_{k})+ g_{2}^{+}(X_{i}, X_{j})g_{1}^{+}(X_{j}, X_{k}) = g_{2}^{+}(X_{i}, X_{k})
\end{aligned}
\]

\begin{thm}\label{Thm-PHiffAlexPair}
    The following are equivalent.
    \begin{itemize}
        \item[(1)] $g=(g_{1}^{+}, g_{1}^{-}, g_{2}^{+}, g_{2}^{-})$ preserves holonomy, that is, (A) through (C) holds.
        \item[(2)] $g=(g_{1}^{+}, g_{1}^{-}, g_{2}^{+}, g_{2}^{-})$ is f-twisted Alexander weight.
    \end{itemize}
\end{thm}

\begin{proof}
  $(1) \Longrightarrow (2)$. Since (B-1) and (B-2) hold, we have
  \[
    \begin{cases}
    g_{1}^{-}(X_{j}, X_{i})g_{1}^{+}(X_{j+1}, X_{i}) = 1,\\
    g_{1}^{+}(X_{i}, X_{j})g_{1}^{-}(X_{i+1}, X_{j}) = 1,\\
    g_{1}^{-}(X_{i}, X_{j})g_{1}^{+}(X_{i+1}, X_{j}) =1.
    \end{cases}
  \]
  Applying $\ast^{-1} X_{j}$ to the second equation, and using $X_{i} * X_{j} = X_{i+1}$, we obtain
  \[
    g^{+}_{1}(X_{i} \ast^{-1}X_{j}, X_{j}) g^{-}_{1}(X_{i}, X_{j}) = 1.
  \]
  We then define
  \[
    f_{1}(X_{i}, X_{j}) := g^{-}_{1}(X_{i} \ast X_{j}, X_{j}).
  \]
  Since $f_{1}(X_{i}, X_{j})$ is invertible, it follows that
  \[
    g^{+}_{1}(X_{i}, X_{j}) = ( g^{-}_{1}(X_{i} \ast X_{j}, X_{j}) )^{-1} = f_{1}(X_{i}, X_{j})^{-1}.
  \]
  Next, we rewrite (B-1) and (B-2) by applying $\ast^{-1} X_{j}$:
  \[
    \begin{cases}
       g^{-}_{2}(X_{j}, X_{i}) + g^{-}_{1}(X_{j}, X_{i})g^{+}_{2}(X_{j} \ast^{-1} X_{i}, X_{i}) = 0,\\
       g^{+}_{2}(X_{i}, X_{j}) + g^{+}_{1}(X_{i}, X_{j})g^{-}_{2}(X_{i} \ast X_{j}, X_{j}) = 0,\\
       g^{-}_{2}(X_{i}, X_{j}) + g^{+}_{1}(X_{i},X_{j}) g^{-}_{2}(X_{i} \ast^{-1} X_{j}, X_{k}) = 0.
    \end{cases}
  \]
  We define
  \[
    f_{2}(X_{i}, X_{j}) := g^{-}_{2}(X_{i} \ast X_{j}, X_{j}).
  \]
  Then, we obtain
  \[
    g^{+}_{2}(X_{i}, X_{j})
    = - g^{+}_{1}(X_{i}, X_{j})g^{-}_{2}(X_{i} \ast X_{j}, X_{j})
    = -f_{1}(X_{i}, X_{j})^{-1}f_{2}(X_{i}, X_{j}).
  \]
  Since \( g_1^{\pm 1}, g_2^{\pm 1} \) are already expressed in terms of \( f_1, f_2 \), verifying that the given conditions hold is straightforward. For the implication \( (2) \Rightarrow (1) \), it suffices to trace the previous argument in reverse.
\end{proof}

\begin{prop}
    For all {\it f-twisted weight}, the matrix-zeta function is determined up to unit in $\mathbb{C}[t^{\pm 1}]$.
\end{prop}

\begin{proof}
    What we have to consider is the case (R1-2). The corresponding graph of right-hand side of (R1-2) is as follows:
    
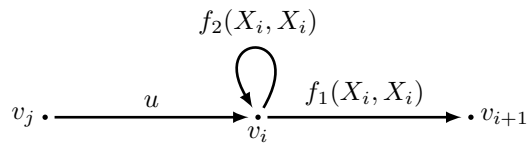
\begin{center}
        \begin{tikzpicture}[x=1.4cm,y=1.4cm]
          \node (vj) at (-2,0) {};
          \node (vi) at (0,0) {};
          \node (vi1) at (2,0) {};
          \node (weight) at (0,0.9) {$f_{2}(X_{i}, X_{i})$};
    
    
          \fill (vj) circle [radius=1pt];
          \fill (vi) circle [radius=1pt];
          \fill (vi1) circle [radius=1pt];
    
          \draw (vj) node [left] {$v_{j}$};
          \draw (vi) node [below] {$v_{i}$};
          \draw (vi1) node [right] {$v_{i+1}$};
    
          \draw[line width = 1pt, arrows={ - latex}] (vj) to node[above]{$u$} (vi);
          \draw[line width = 1pt, arrows={ - latex}] (vi) to node[above]{$f_{1}(X_{i}, X_{i})$} (vi1);
          \path[line width = 1pt, arrows={ - latex}] (vi) edge [out=90-30, in=90+30, distance=12mm] (vi);
    \end{tikzpicture}
  \end{center}
    Since $f_{2}(X_{i}, X_{i}) = 1-f_{i}(X_{i}, X_{i})$ and $f_{1}(X_{i}, X_{i})$ is invertible, $f_{2}(X_{i}, X_{i}) \neq 1$. Then $i-$th column of $I-$(Adjacency matrix) is 
  \[
    \begin{bmatrix}
      \cdots & 1-f_{2}(X_{i}, X_{i}) & -f_{1}(X_{i}, X_{i}) &\cdots
    \end{bmatrix}
    = 
    \begin{bmatrix}
      \cdots & f_{1}(X_{i}, X_{i}) & -f_{1}(X_{i}, X_{i}) &\cdots
    \end{bmatrix}.
  \]
  The other elements of $i-$th column are zero, and multiplying $f_{1}(X_{i}, X_{i})$ for $i$-th column of $I-$(Adjacency matrix), it becomes
  \[
    \begin{bmatrix}
      \cdots & 1 & -1 &\cdots
    \end{bmatrix}
  \]
  The corresponding graph is as follows.
  \begin{center}
    \begin{tikzpicture}[x=1.4cm,y=1.4cm]
      \node (vj) at (-2,0) {};
      \node (vi) at (0,0) {};
      \node (vi1) at (2,0) {};


      \fill (vj) circle [radius=1pt];
      \fill (vi) circle [radius=1pt];
      \fill (vi1) circle [radius=1pt];

      \draw (vj) node [left] {$v_{j}$};
      \draw (vi) node [below] {$v_{i}$};
      \draw (vi1) node [right] {$v_{i+1}$};

      \draw[line width = 1pt, arrows={ - latex}] (vj) to node[above]{$u$} (vi);
      \draw[line width = 1pt, arrows={ - latex}] (vi) to node[above]{$1$} (vi1);

    \end{tikzpicture}
  \end{center}
  Then the difference of the zeta functions before and after the (R1-2) is $f_{1}(X_{i}, X_{i})^{-1} \in \mathbb{C}[t^{\pm 1}]^{\times}$
\end{proof}

\begin{rem}
As mentioned in Proposition \ref{prop Alex-pair-cond}, since \( f = (f_1, f_2) \) is an Alexander pair, the set \( Q \times R \) inherits a quandle structure via the operation
\[
(a, x) \mathbin{\star} (b, y) = (a \ast b, f_1(a, b)x + f_2(a, b)y).
\]
This can be interpreted geometrically: we may think of \( Q \times R \) as a kind of covering space over \( Q \), with \( R \) as the fiber. The operation then describes how an element in the fiber moves when transported along a path in the base space.

In this setting, we interpret a cycle in the graph as an element of its fundamental group. Lifting such a cycle to the covering space \( Q \times R \), the difference between the starting and ending points in the fiber reflects the holonomy associated with that cycle. The matrix \( I - w(C) \) thus records this holonomy, capturing the accumulated effect of moving through the cycle \( C \).

A holonomy preserving transformation, then, ensures that equivalent cycles lift to paths with the same endpoints in \( Q \times R \), thereby preserving the holonomy structure of the graph.
\hfill \( \square \)
\end{rem}

\providecommand{\bysame}{\leavevmode\hbox to3em{\hrulefill}\thinspace}
\providecommand{\MR}{\relax\ifhmode\unskip\space\fi MR }
\providecommand{\MRhref}[2]{%
  \href{http://www.ams.org/mathscinet-getitem?mr=#1}{#2}
}
\providecommand{\href}[2]{#2}

\end{document}